\documentclass{article}
\usepackage{geometry}
\usepackage[T1]{fontenc}
\usepackage[utf8]{inputenc}
\usepackage{fourier}
\usepackage{parskip}
\usepackage{amsfonts}
\usepackage{amsmath}

\usepackage{booktabs} 
\usepackage{array} 
\usepackage{paralist} 
\usepackage{verbatim} 
\usepackage{subfig} 
\usepackage{amscd}
\usepackage{amsthm}
\usepackage{amssymb}
\usepackage{xcolor}
\usepackage{framed}
\usepackage{soul}
\usepackage{verbatim}

\usepackage{tikz-cd}
\usetikzlibrary{matrix,arrows,decorations.pathmorphing}

\usepackage{enumitem}

\usepackage{cancel}
\usepackage{fullpage}
\usepackage{halloweenmath}
\usepackage{rsfso}
\usepackage{hyperref}
\hypersetup{
    colorlinks=true,
    linkcolor=blue,
    filecolor=magenta,
    urlcolor=cyan,
    citecolor=blue,
}

\DeclareMathOperator{\HEIGHT}{ht}

\DeclareMathOperator{\ANN}{ann}

\DeclareMathOperator{\SPEC}{Spec}
\DeclareMathOperator{\QUOT}{Quot}
\DeclareMathOperator{\ADJ}{Adj}
\DeclareMathOperator{\PFADJ}{PfAdj}
\DeclareMathOperator{\RANK }{rank}

\DeclareMathOperator{\SUPP}{Supp}
\DeclareMathOperator{\TOR}{Tor}
\DeclareMathOperator{\EXT}{Ext}
\DeclareMathOperator{\DEPTH }{depth}
\DeclareMathOperator{\GRADE }{grade}
\DeclareMathOperator{\PDIM}{pd}

\DeclareMathOperator{\SYM}{Sym}

\DeclareMathOperator{\FITT}{Fitt}
\DeclareMathOperator{\PFAFF}{Pf}

\DeclareMathOperator{\REG}{reg}
\DeclareMathOperator{\TOPDEG}{topdeg}
\DeclareMathOperator{\CHAR}{char}
\DeclareMathOperator{\PROJ}{Proj}
\DeclareMathOperator{\trdeg}{trdeg}

\usetikzlibrary{positioning}

\mathchardef\mhyphen="2D

\usepackage[nameinlink,capitalize]{cleveref}

\newcommand{\N}{\mathbb{N}}

\newcommand{\ann}[2]{\ANN_{#1}\!\left(#2\right)} 

\newcommand{\mf}[1]{\mathfrak{#1}}
\newcommand{\mcal}[1]{\mathcal{#1}}
\newcommand{\fm}{\mathfrak{m}}
\newcommand{\fn}{\mathfrak{n}}
\newcommand{\fp}{\mathfrak{p}}

\newcommand{\chr}[1]{\CHAR #1}

\newcommand{\eltlist}[3]{ #1_{#2}, \ldots, #1_{#3} }

\newcommand{\poly}[2]{#1 \! \left[ #2 \right]}

\newcommand{\polySmall}[2]{#1 \left[ #2 \right]}

\newcommand{\spec}[1]{\SPEC\!\left(#1\right)} 
\newcommand{\proj}[1]{\PROJ\!\left(#1\right)}
\newcommand{\quot}[1]{\QUOT\!\left(#1\right)} 
\newcommand{\adj}[1]{\ADJ\!\left(#1\right)} 
\newcommand{\pfadj}[1]{\PFADJ\!\left(#1\right)} 
\newcommand{\rank}[1]{\RANK\!\left(#1\right)} 
\newcommand{\gens}[2]{\mu_{#1}\!\left(#2\right)}

\newcommand{\wtld}[1]{\widetilde{#1}}

\newcommand{\supp}[2]{\SUPP_{#1}\!\left(#2\right) } 

\newcommand{\Tor}[4]{\TOR_{#1}^{#2}\!\left(#3, #4\right)} 
\newcommand{\Ext}[4]{\EXT_{#2}^{#1}\!\left(#3, #4\right)} 
\newcommand{\comp}[2][\bullet]{\mathbf{#2}\!_{#1}}
\newcommand{\compp}[2][\bullet]{\mathbf{#2}_{#1}}

\newcommand{\depth}[2]{\DEPTH_{#1}  #2 } 
\newcommand{\grade}[1]{\GRADE  #1 } 
\newcommand{\pdim}[2]{\PDIM_{#1}\!  #2 } 
\newcommand{\hgt}[1]{\HEIGHT  #1 } 
\newcommand{\rees}[1]{\mathcal{R}\!\!\left(#1\right)}
\newcommand{\sym}[2]{\SYM_{#1}\!\left(#2\right)} 

\newcommand{\locoh}[3]{\mathrm{H}^{#1}_{#2}\hspace{-0.7mm}\left(#3\right)}

\newcommand{\aspread}[1]{\ell \! \left( #1 \right)}
\newcommand{\detid}[2]{I_{#1}\!\left(#2\right)}
\newcommand{\pfaffid}[2]{\PFAFF_{#1}\!\left(#2\right)}

\newcommand{\fitt}[2]{\FITT_{#1}\!\left( #2 \right)}

\newcommand{\td}[1]{\TOPDEG #1}
\newcommand{\tdp}[1]{\TOPDEG \! \left( #1 \right)}

\newcommand{\A}[2]{\mcal{A}_{#1}\!\left(#2\right)}
\newcommand{\bzero}[1]{b_{0}\!\left(#1\right)}
\newcommand{\fiber}[1]{\mathcal{F}\hspace{-1mm}\left(#1\right)}
\newcommand{\hmgy}[2]{\mathrm{H}_{#1}\hspace{-0.7mm}\left(#2\right)}

\newcommand{\inlineqt}[2]{\left. #1 \!\middle/ #2 \right.}

\newcommand{\paren}[1]{\left(#1\right)}

\newcommand{\brces}[1]{\left\{#1\right\}}

\newtheorem{thm}{Theorem}
\newtheorem{prop}[thm]{Proposition}
\newtheorem{lem}[thm]{Lemma}
\newtheorem{cor}[thm]{Corollary}

\newtheorem{data}[thm]{Data}
\theoremstyle{definition}
\newtheorem{notation}[thm]{}

\numberwithin{thm}{section}

\usepackage{etoolbox}
\apptocmd{\sloppy}{\hbadness 10000\relax}{}{}

\colorlet{shadecolor}{orange!20}

\title{Bounding the degrees of the defining equations of Rees rings for certain determinantal and Pfaffian ideals}
\author{Monte Cooper and Edward F. Price III}
\date{}

\begin{document}

\maketitle

\begin{abstract}
    We consider ideals of minors of a matrix, ideals of minors of a symmetric matrix, and ideals of Pfaffians of an alternating matrix. Assuming these ideals are of generic height, we characterize the condition $G_{s}$ for these ideals in terms of the heights of other ideals of minors or Pfaffians of the same matrix. We additionally obtain bounds on the generation and concentration degrees of the Rees rings of a subclass of such ideals via specialization of the Rees rings in the generic case. We do this by proving, given sufficient height conditions on ideals of minors or Pfaffians of the matrix, the specialization of a resolution of the Rees ring in the generic case is an approximate resolution of the Rees ring in question. We end the paper by giving some explicit generation and concentration degree bounds. 
\end{abstract}

\large

\section{Introduction}

Let $R$ be a Noetherian ring and $I$ be an $R$-ideal. The \textit{Rees ring} of $I$ is $\rees{I} = R\!\left[It\right] \subseteq R\!\left[t\right]$, where $t$ is an indeterminate over $R$. Moreover, the Rees ring may be regarded as a standard graded ring over $R$, and note $\rees{I} \cong \bigoplus_{i =0}^{\infty} I^{i}$. Many invariants and properties of concern in commutative algebra depend on understanding the asymptotic properties of the powers of an ideal $I$. For example, information about the various notions of multiplicity of an ideal $I$ can be obtained by understanding the asymptotic properties of the ideal. Further, the normality of an ideal can be detected using the Rees ring. Since the Rees ring encodes information about all powers of the ideal $I$, a great deal of information can be obtained through the study of the Rees ring.

Understanding the Rees ring is important to other disciplines as well. For example, in algebraic geometry, if $R$ is the coordinate ring of an affine variety, and $I$ is the ideal corresponding to a subvariety, then $\rees{I}$ is the homogeneous coordinate ring of the blowup of the variety along the subvariety.  Moreover, if $R = K\!\left[x_{1},\ldots,x_{d}\right]$ is a standard graded polynomial ring over a field $K$ and $I = \left(f_{1},\ldots,f_{n}\right)$ is generated by $n$ homogeneous polynomials of the same degree, one can define a rational map between projective spaces $\Psi \colon \mathbb{P}_{K}^{d-1} \dashrightarrow \mathbb{P}_{K}^{n-1}$ away from $V\!\left(I\right)$ via $\Psi\!\left(\left[a_{1} : \: \cdots \: : a_{d} \right]\right) = \left[f_{1}\!\left(a_{1},\ldots,a_{d}\right) : \: \cdots \: : f_{n}\!\left(a_{1},\ldots,a_{d}\right)\right]$. In this setting, the Rees ring $\rees{I}$ is the bihomogeneous coordinate ring of the graph of $\Psi$.

The descriptions of the Rees ring given above provide parametric equations for $\rees{I}$; however, if one wishes to study invariants of an ideal or geometric properties of the blowup or of a rational map, obtaining implicit equations is much more useful. Indeed, there has been interest in studying the implicit defining equations of $\rees{I}$ and similar algebras in the areas of geometric modeling \cite{Cox} and chemical reaction networks \cite{CLS}. 

Much work has been done in studying the implicit defining equations of $\rees{I}$ in a few cases. Considerable work has been done in the case where $I$ is perfect of grade two \cite{HSV} \cite{Morey} \cite{Morey-Ulrich} \cite{CHW} \cite{HSV2} \cite{Buse} \cite{KPURNS} \cite{CBDA} \cite{CBDA2} \cite{Nguyen} \cite{Madsen} \cite{BM} \cite{KPUBGA} \cite{Sammartano} \cite{Kim-Mukundan}. Some work has also been done in the case where $I$ is a perfect Gorenstein ideal of grade three \cite{Morey} \cite{Johnson} \cite{KPU}. More generally, there has been some work in the cases where $I$ is the ideal of minors of a matrix; in particular, \cite{BCV2} and \cite{HPPRS} study this problem for ideals of minors of a generic matrix, and \cite{BCV} studies this problem for matrices with linear entries.  

It is the aim of this paper to obtain degree bounds on the implicit defining equations of $\rees{I}$ when $I$ is the ideal of minors of a matrix or symmetric matrix or when $I$ is the ideal of Pfaffians of an alternating matrix. We do not require that the matrices are generic nor do we require the entries to be linear forms. We obtain these bounds by specializing from the generic case and controlling the specialization map between the respective Rees rings. In order to control the specialization map, we require the height of $I$ to be the same as the corresponding ideal in the generic case.  In order to obtain degree bounds, we often require the ideals of lower size minors to satisfy certain height bounds, which are stated where relevant.

\Cref{SectionNotation} describes notation and conventions assumed throughout the paper, in particular for matrices, determinantal ideals, resolutions and approximate resolutions, Fitting ideals and the condition $G_{s}$, the Rees ring and its defining equations, the symmetric algebra, and the special fiber ring.

Since the condition $G_{s}$ is often used in the study of Rees rings, \cref{SectionGConditions} focuses on these conditions. We characterize $G_{s}$ for generic height determinantal and Pfaffian ideals of a matrix $A$ in \cref{GsConditions}. This characterization is made in terms of ideals of lower size minors (or lower size Pfaffians) of $A$, giving a more computationally efficient method to verify $G_{s}$ than through the use of Fitting ideals. We use this characterization to provide a classification of the maximal $s$ for which determinantal and Pfaffian ideals of a generic matrix satisfy $G_{s}$ in \cref{GsGeneric}, \cref{GsSymmetric}, and \cref{GsAlternating}.

\Cref{SectionSpecialization} is concerned with the specialization process. We describe how to specialize resolutions of generic determinantal ideals without explicit reference to the differentials in a given resolution. We also establish the main tools used to control the defining ideal of $\rees{I}$ in terms of the generic case, namely, \cref{SpecializationLemma} and \cref{ContainmentLemma}. We finish the section by using the main tools to establish when the Rees ring specializes in \cref{SpecializationRees} and classifying certain determinantal and Pfaffian ideals as being of linear type or of fiber type in \cref{LinFibCor}. 

Lastly, in \cref{SectionDegreeBounds}, we give bounds on the generation and concentration degrees of the defining equations for $\rees{I}$. 

In particular, \cref{OrdSection} develops degree bounds for ordinary matrices. We prove bounds for maximal minors in \cref{TheoremMaximalMinors}. Using this, we recover results of \cite[3.7]{BCV} and \cite[6.1.a]{KPU}. We address the case of $2 \times 2$ minors in \cref{Theorem2x2Minors} and show the $2 \times 2$ minors of a $3 \times n$ matrix with linear entries is of fiber type in \cref{FiberTypeSize2}. In \cref{TheoremSubmaximalMinorsOrdinary}, we obtain bounds in the case of submaximal minors of a square matrix, and we give degree bounds generally in \cref{TheoremArbitraryMinors}.

For \cref{SymSection}, we obtain general results for submaximal minors of a symmetric matrix in \cref{SymmetricSubmaximal}; however, we are unable to produce explicit degree bounds due to a lack of known results for ideals of minors of symmetric matrices.

Finally, in \cref{AltSection}, we prove degree bounds for Pfaffian ideals of alternating matrices. We recover the results of \cite[6.1b]{KPU} in \cref{PfaffianDegreeBoundFacts}. We obtain new results about size $n-2$ Pfaffians in \cref{TheoremN-2Pfaffs} and show $4 \times 4$ Pfaffians of a $6 \times 6$ alternating matrix with linear entries is of fiber type.  We give degree bounds for $4\times 4$ Pfaffians in \cref{Theorem4Pfaffs}, and we finish with general degree bounds in \cref{TheoremGeneralPfaffs}.

\section{Notation, Conventions, and Background}
\label{SectionNotation}
\hypertarget{SectionNotation}{} \:

\begin{notation}\label{MatrixNotation}\hypertarget{MatrixNotation} Let $R$ be a Noetherian ring, and let $A$ be an $m \times n$ matrix with entries in $R$. Throughout this paper, we use the convention $m \le n$. We use $\detid{t}{A}$ to denote the ideal generated by the $t \times t$ minors of $A$ in the case $1 \le t \le m$. By convention, if $t \le 0$, we set $\detid{t}{A} = R$, and if $t > m$, we set $\detid{t}{A} = 0$.   

By a \textit{generic matrix} $X$ over $R$, we mean a matrix whose entries are distinct indeterminates over $R$.

It is well-known \cite[Theorem 3]{EN} that if $\detid{t}{A}$ is proper, then $\hgt{\detid{t}{A}} \le \left(m-t+1\right)\left(n-t+1\right)$, and equality is achieved if $A$ is a generic matrix \cite[2]{Eagon}. As such, the ideal $\detid{t}{A}$ is said to be \textit{of generic height} if it achieves this maximal height, i.e., if $\,\hgt{\detid{t}{A}} = \left(m-t+1\right)\left(n-t+1\right)$.

For any $m \times n$ matrix $A$, the notation $A^{T}$ refers to the \textit{transpose} of $A$.  We say is $A$ \textit{symmetric} if $A^{T} = A$. We say $A$ is \textit{alternating} if $A^{T} = -A$ and if the diagonal entries of $A$ are $0$. 

By a \textit{generic symmetric matrix} $X$ over $R$, we mean a symmetric matrix whose upper triangle consists of distinct indeterminates over $R$. It is known \cite[2.1]{Jozefiak} that if $A$ is an $n \times n$ symmetric matrix and if $\detid{t}{A}$ is a proper ideal, then $\hgt{\detid{t}{A}} \le \binom{n-t+2}{2}$, and equality is achieved if $A$ is a generic symmetric matrix. As such, if $A$ is an $n \times n$ symmetric matrix, then the ideal $\detid{t}{A}$ is said to be \textit{of generic symmetric height} if it achieves this maximal height, i.e., if $\hgt{\detid{t}{A}} = \binom{n-t+2}{2}$.

It is well-known that if $A$ is an $n \times n$ alternating matrix, then $\det A$ is a perfect square in $R$. As such, one considers the  \textit{Pfaffian} of an alternating matrix $A$, denoted $\pfaffid{}{A}$, as a polynomial in the entries of $A$ whose square is $\det A$. For a definition and a more in depth discussion of the properties of Pfaffians, see \cite[Appendix D]{FP} or \cite[pp. 140-142]{Artin}. When $n$ is odd, $\det A = 0$, so we focus on submatrices of even size. Hence, we shall consider submatrices of size $2t \times 2t$ instead of $t \times t$ as in the case of ideals of minors.

Given an $n \times n$ alternating matrix $A$, by $\pfaffid{2t}{A}$, we denote the ideal generated by the Pfaffians of the $2t \times 2t$ principal submatrices of $A$ in the case $2 \le 2t \le n$. By convention, if $2t \le 0$, then $\pfaffid{2t}{A} = R$, and if $2t > n$, then $\pfaffid{2t}{A} = 0$. For simplicity, we also refer to the Pfaffians of the $2t \times 2t$ principal submatrices of $A$ as the $2t \times 2t$ Pfaffians of $A$.

By a \textit{generic alternating matrix} $X$ over $R$, we mean an alternating matrix whose entries above the diagonal are distinct indeterminates over $R$. It is known \cite[2.1, 2.3]{JP} that if $A$ is an $n \times n$ alternating matrix and if $\pfaffid{2t}{A}$ is a proper ideal, then $\hgt{\pfaffid{2t}{A}} \le \binom{n-2t+2}{2}$, and equality is achieved if $A$ is a generic alternating matrix. As such, if $A$ is an $n \times n$ alternating matrix, then the ideal $\pfaffid{2t}{A}$ is said to be \textit{of generic alternating height} if it achieves this maximal height, i.e., if $\hgt{\pfaffid{2t}{A}} = \binom{n-2t+2}{2}$.

Throughout this paper we shall provide data intended to establish the setting corresponding to each of the above types of matrices. In particular, in our given data, we shall always use (a) to refer to $m \times n$ matrices while (b) and (c) shall refer to the symmetric and alternating cases, respectively.
\end{notation}

\bigskip

\begin{notation}
\label{AdjNotation}\hypertarget{AdjNotation}
Let $A$ be an $n \times n$ matrix with entries in a ring $R$. It is a classical result that there exists an $n \times n$ matrix $\adj{A}$, called the \textit{classical adjoint} of $A$, with entries in $R$ satisfying the property $\adj{A}A = \det\!\left(A\right)I_{n \times n}$. This follows from the Laplace expansion of $\det A$.

Analogously, let $n$ be even, and let $A$ be an $n \times n$ alternating matrix with entries in $R$. Then there exists an $n \times n$ alternating matrix $\pfadj{A}$, called the \textit{Pfaffian adjoint} of $A$, with entries in $R$ satisfying the property $\pfadj{A}A = \pfaffid{}{A} I_{n \times n}$. As in the non-alternating case, this result follows from a Laplace expansion of the Pfaffian. In particular, if $A$ is an $n \times n$ alternating matrix with $n$ even, then
\[ \pfaffid{}{A} = \sum\limits_{i < j} \left(-1\right)^{i+j-1}A_{ij}\mathrm{Pf}^{\,ij}\hspace{-1mm}\left(A\right) + \sum\limits_{i > j} \left(-1\right)^{i+j}A_{ij}\mathrm{Pf}^{\,ij}\hspace{-1mm}\left(A\right),\]
where $\mathrm{Pf}^{\,ij}\hspace{-1mm}\left(A\right)$ denotes the Pfaffian of the alternating matrix obtained from $A$ by deleting rows and columns $i$ and $j$ (see, for instance, \cite[Appendix D]{FP}).  Define $\pfadj{A}$ to be the $n \times n$ alternating matrix whose $\left(i,j\right)$-entry is $\left(-1\right)^{i+j}\mathrm{Pf}^{\,ij}\hspace{-1mm}\left(A\right)$ for $i < j$. Using the Laplace expansion above, one verifies  $\pfadj{A} A = \pfaffid{}{A} I_{n \times n}$.
\end{notation}

\bigskip

\begin{notation}\label{BoundsNotation}\hypertarget{BoundsNotation}
Given a graded Noetherian ring $R$ and a graded $R$-Module $M$ we define the \textit{generation degree} of $M$, $\bzero{M}$, and the \textit{concentration degree} of $M$, $\td{M}$, as follows:

\begin{align*}
     \bzero{M} &=  \inf \left\{ p \:\middle\vert\: R\left( \bigoplus_{j \le p} \left[M\right]_{j} \right) = M \right\},  \text{ and}\\
     \td M &= \sup \left\{ j \:\middle\vert\: \left[M\right]_{j} \neq 0 \right\}.& 
\end{align*}

Note that if $M = 0$, then we have \[ \tdp{M} = \bzero{M} = -\infty.\]
\end{notation}

\bigskip

\begin{notation}\label{RegularityNotation}\hypertarget{RegularityNotation} Given a standard graded polynomial ring $R$ over the field $K$ with maximal homogeneous ideal $\fm$ and a finitely generated graded $R$-module $M$, let $\left( \compp{C}, \comp{\partial} \right)$ be a minimal homogeneous free resolution of $M$. The Castelnuovo-Mumford \textit{regularity} of $M$ is 
\[ \REG M = \sup \left\{ \bzero{\comp[j]{C}} - j\right\}.\]
\end{notation}

\bigskip

\begin{notation}\label{BECriterionNotation}\hypertarget{BECriterionNotation} Let $R$ be a Noetherian ring, given a complex 
\[ \comp F \colon 0 \to F_{s} \overset{\varphi_{s}}{\longrightarrow} F_{s-1} \to \cdots \to F_{1} \overset{\varphi_{1}}{\longrightarrow} F_{0} \to 0  \] of finite free $R$-modules, we set $r_{i} = \sum_{j=i}^{s}\left(-1\right)^{j-i} \rank{F_{j}}$ and use $I \! \left( \varphi_{i} \right)$ to denote $\detid{r_{i}}{\varphi_{i}}$.
\end{notation}

\bigskip

\begin{notation}\label{ApproximateResolutionNotation}\hypertarget{ApproximateResolutionNotation} Let $R$ be a Noetherian nonnegatively graded ring of dimension $d > 0$ with $R_{0}$ local and with unique maximal homogeneous ideal $\fm$, and suppose $M$ is a graded $R$-module. 
For any homogeneous complex of finitely generated graded modules $\compp D$ with $M \cong \hmgy{0}{\compp D}$, we say $\compp D$ is an \textit{approximate resolution} of $M$ if both of the following conditions hold:
\begin{enumerate}
    \item[$(\mathrm{i})$] $\dim \hmgy{j}{\compp{D}} \le j$ whenever $1 \le j \le d-1$, and
    \item[$(\mathrm{ii})$] $\min\!\left\{ d, j+2\right\} \le \depth{}{\compp[j]{D}}$ whenever $0 \le j \le d-1$.
\end{enumerate}

While we define this in terms of graded rings and modules, we note this definition also applies to local rings by giving the trivial grading.

We will focus on complexes of free modules; hence, condition (ii) will be automatically satisfied if $R$ is Cohen-Macaulay.
\end{notation}

\bigskip

\begin{notation}\label{PerfectNotation}\hypertarget{PerfectNotation} When $R$ is a Noetherian ring, $I$ is a proper $R$-ideal, and $M$ is a nonzero finitely generated $R$-module, it is always the case that $\grade{\!\left(\ann{R}{M}\right)} \le \pdim{R}{M}$, where $\pdim{R}{M}$ denotes the projective dimension of $M$ as an $R$-module. If equality holds we say $M$ is a \textit{perfect} $R$-module. In the ideal case, though it has the structure of an $R$-module, we say $I$ is a \textit{perfect ideal} when the $R$-module $R/I$ is perfect.

Given a perfect $R$-ideal $I$ with $\grade{\!\left(I\right)} = g\,$ we say $I$ is a \textit{Gorenstein ideal} if $\Ext{g}{R}{R/I}{R}$ is a cyclic $R$-module.
\end{notation}

\bigskip

\begin{notation}\label{GsNotation}\hypertarget{GsNotation} For a ring $R$ and an $R$-module $M$ we use the notation $\tau_{R}\!\left(M\right)$ to denote the $R$-torsion of $M$. We use the notation $\gens{R}{M}$ to refer to the minimal number of generators of a finitely generated module $M$ over the local ring $R$. 

Given an ideal $I$ in a Noetherian ring $R$, $I$ is said to satisfy the condition $G_{s}$ if $\,\gens{R_{\fp}}{I_{\fp}} \le \dim R_{\fp}$ for all $\fp \in V\!(I)$ with $\dim R_{\fp} \le s-1$.

We make use of an equivalent description of the condition $G_{s}$ in terms of heights of Fitting ideals. In particular, for an ideal $I$ with $\hgt{I} > 0$, $I$ satisfies $G_{s}$ if and only if $\hgt{\fitt{i}{I}} > i$ for $0 < i < s$ if and only if $\hgt{\fitt{i}{I}} \ge \min\!\left\{i+1,s\right\}$ for all $i \ge 1$.
\end{notation}

\bigskip

\begin{notation}\label{ReesNotation}\hypertarget{ReesNotation} Let $R$ be a Noetherian ring, and suppose $I$ is an $R$-ideal. The \textit{Rees ring} of $I$ over $R$ is defined to be the $R$-subalgebra $R\!\left[It\right] \subseteq R\!\left[t\right]$, where $t$ is an indeterminate over $R$, and is denoted $\rees{I}$. We note $\rees{I}$ is a standard graded $R$-algebra whose grading is induced by the standard grading on $R\!\left[t\right]$. 

Given $R = K\!\left[x_{1},\ldots,x_{d}\right]$ a standard graded polynomial ring over a field $K$ and $I = \left(f_{1},\ldots,f_{n}\right)$ generated by homogeneous forms $f_{1}, \ldots, f_{n}$ of the same degree $D$, we define $S = R\!\left[T_{1},\ldots,T_{n}\right]$ as a standard bigraded $K$-algebra where $\deg x_{i} = \left(1,0\right)$ and $\deg T_{i} = \left(0,1\right)$. We give $\rees{I}$ the bigrading where $\deg x_{i} = \left(1,0\right)$ and $\deg t = \left(-D,1\right)$ in order to force the $R$-algebra homomorphism $\pi \colon S \to \rees{I}$ given by $T_{i} \mapsto f_{i}t$ to be a bihomogeneous $K\!$-algebra epimorphism. We use $\mathcal{J}$ to denote the bihomogeneous ideal $\ker \pi$ and call this the \textit{defining ideal} of $\rees{I}$. An element of a minimal bihomogeneous generating set of $\mathcal{J}$ is referred to as a \textit{defining equation} of $\rees{I}$.
\end{notation}

\bigskip

\begin{notation}\label{SymmetricNotation}\hypertarget{SymmetricNotation} We also note $\sym{}{I\!\left(D\right)}$ is a standard bigraded $K\!$-algebra, and the natural $R$-algebra homomorphism $\alpha \colon \sym{}{I\!\left(D\right)} \to \rees{I}$ given by $f_{i} \mapsto f_{i}t$ is a bihomogeneous $K\!$-algebra epimorphism. We use $\mathcal A$, or $\A{}{I}$, to denote the bihomogeneous ideal $\ker \alpha$. We use the notation $\A{}{I}$ primarily to distinguish between multiple ideals under consideration. It is important to recognize that  the $R$-algebra homomorphism $f \colon S \to \sym{}{I\!\left(D\right)}$ with $T_{i} \mapsto f_{i}$ is a bihomogeneous $K\!$-algebra epimorphism.  We use $\mathcal{L}$ to denote the bihomogeneous ideal $\ker f$. We then obtain the following commutative diagram with exact rows in which all maps are bihomogeneous.

\[  \begin{tikzcd}
    0 \ar[r] & \mathcal{L} \ar[r] \ar[d, hookrightarrow] & S \ar[d, equal] \ar[r, "f"] & \sym{}{I\!\left(D\right)} \ar[r] \ar[d, "\alpha"] & 0\\
    0 \ar[r] & \mathcal{J} \ar[r] & S \ar[r, "\pi"] & \rees{I} \ar[r] & 0
    \end{tikzcd} \]

As such, one sees $\mathcal{A} = \inlineqt{\mathcal{J}}{\mathcal{L}}$.  Moreover, it is known that $S\left[\mathcal{J}\right]_{\left(*,1\right)} =\mathcal{L}$. As such, $\left[\mathcal{A}\right]_{\left(*,1\right)} = 0$. Moreover, since $\mathcal{L}$ is easy to find through a presentation matrix of $I$, we devote our attention to $\mathcal{A}$.
\end{notation}

\bigskip

\begin{notation}\label{AIdealNotation}\hypertarget{AIdealNotation} We adopt the notation $\sym{k}{I\!\left(D\right)} = \left[\sym{}{I\!\left(D\right)}\right]_{\left(*,k\right)}$. In the same vein, we use the notation $\A{k}{I} = \left[\A{}{I}\right]_{\left(*,k\right)}$.  As such we write $\left[\A{k}{I}\right]_{p}$ to mean $\left[\A{}{I}\right]_{\left(p,k\right)}$. This notation is convenient for us since we often fix the second component in the bidegree. 
\end{notation}

\bigskip

\begin{notation}\label{LinearTypeNotation}\hypertarget{LinearTypeNotation} An ideal $I$ is of \textit{linear type} if $\A{}{I} = 0$, or equivalently, if $\bzero{\A{k}{I}} = \td{\A{k}{I}} = -\infty$ for all $k$. The ideal $I$ is of \textit{fiber type} if $\A{k}{I}$ is generated in degree zero for all $k$, or equivalently, $\bzero{\A{k}{I}} \le 0$ for all $k$.
\end{notation}

\bigskip

\begin{notation}\label{AnalyticSpreadNotation}\hypertarget{AnalyticSpreadNotation} Suppose $\left(R,\fm\right)$ is a Noetherian standard graded ring with unique maximal homogeneous ideal $\fm$, and $I$ is a proper $R$-ideal generated by forms of the same degree. The \textit{special fiber ring} of $I$, denoted $\fiber{I}$, is defined as $\fiber{I} = \inlineqt{\rees{I}}{\fm\rees{I}}$. The \textit{analytic spread} of $I$ is defined as $\dim \fiber{I}$, and is denoted $\aspread{I}$. It is the case that $\aspread{I} \le \min\!\left\{ \gens{}{I}, \dim R \right\}$. Burch \cite{Burch} proved $\aspread{I} + \inf\!\left\{\depth{R} R/I^{k}\right\} \le \dim R$. 
\end{notation}

\bigskip

\section{The Conditions \texorpdfstring{$G_{s}$}{}}
\label{SectionGConditions}
\hypertarget{SectionGConditions}{} 

In much of the work that has been done on studying Rees rings of an ideal $I$, the ideal $I$ has been assumed to satisfy the condition $G_{s}$ for some $s$. One finds $G_{s}$ is often important for the study of Rees rings because of the relation of the condition to the dimension of the symmetric algebra of $I$. Specifically, when $R$ is a Noetherian ring and $I$ is a proper $R$-ideal, \cite{Huneke-Rossi} shows $\dim \sym{}{I} = \sup\!\left\{ \dim R/\fp + \gens{R_{\fp}}{I_{\fp}} \:\middle\vert\: \fp \in \spec{R} \right\}$.

The typical way to compute whether an ideal $I$ satisfies the condition $G_{s}$ for some $s$ is by computing the heights of the Fitting ideals $\fitt{j}{I}$. To compute $\fitt{j}{I}$, one needs to construct a presentation matrix $\varphi$ of $I$ and compute ideals of minors of $\varphi$. The following proposition allows us to bypass both the use of Fitting ideals and the construction of $\varphi$ when working with determinantal ideals $I = \detid{t}{A}$ by characterizing $G_{s}$ in terms of the matrix $A$ itself.

\bigskip

\begin{prop}
\label{GsConditions}
\hypertarget{GsConditions}{} Let $R$ be a Noetherian ring.
\begin{enumerate}
    \item[$(\mathrm{a})$]\label{GsConditionsOrd}\hypertarget{GsConditionsOrd}{} Suppose $ 1 \le t \le m \le n$, and let $A$ be an $m \times n$ matrix with entries in $R$.  Suppose $I = \detid{t}{A}$ is of generic height.  Then $I$ satisfies $\,G_{s}$ if and only if 
\[ \hgt{\detid{j}{A}} \ge \min\left\{\binom{m-j+1}{m-t}\binom{n-j+1}{n-t}, s\right\} \text{\:  for all } \,1 \le j \le t-1. \]
    
    \item[$(\mathrm{b})$]\label{GsConditionsSym}\hypertarget{GsConditionsSym}{} Suppose $1 \le t \le n$, and $A$ is a symmetric $n \times n$ matrix. Let $I = \detid{t}{A}$ be of generic symmetric height. Then $I$ satisfies $\,G_{s}$ if and only if

\[ \hgt{\detid{j}{A}} \ge \min\!\left\{ \frac{1}{n-j+2}\binom{n-j+2}{n-t}\binom{n-j+2}{n-t+1}, s \right\} \text{\:  for all } \,1 \le j \le t-1. \]
   
    \item[$(\mathrm{c})$]\label{GsConditionsAlt}\hypertarget{GsConditionsAlt}{} Suppose $2 \le 2t \le n$, and $A$ is an alternating $n \times n$ matrix. Let $I = \pfaffid{2t}{A}$ be of generic alternating height. Then $I$ satisfies $\,G_{s}$ if and only if

\[ \hgt{\pfaffid{2j}{A}} \ge \min\!\left\{ \binom{n-2j+2}{n-2t}, s \right\} \text{ \:  for all } \,1 \le j \le t-1. \]
\end{enumerate}
\end{prop}

\begin{proof}
Let $\left(\mu_{j}\right)_{j=0}^{q}$ be a strictly decreasing sequence of positive integers satisfying the following properties
\begin{enumerate}
    \item[(i)]  $I$ can be generated by $\mu_{0}$ elements, 
    \item[(ii)] $\mu_{q} = \hgt{I} > 0$, and
    \item[(iii)] $\left\{ \gens{R_{\fp}}{I_{\fp}} \:\middle\vert\:  \fp \in V\!\left(I\right)\right\} \subseteq \left\{ \mu_{j} \:\middle\vert\: 0 \le j \le q  \right\} $.
\end{enumerate}

Recall from \cref{GsNotation} that if $\hgt{I} > 0$, then $I$ satisfies $G_{s}$ if and only if \[\hgt{\fitt{k}{I}} \ge \min\!\left\{ k+1, s\right\} \text{ for all } k \ge 1.\] 
We begin by showing $I$ satisfies $G_{s}$ if and only if 
\[\hgt{\fitt{\mu_{j}}{I}} \ge \min\!\left\{\mu_{j-1},s\right\} \text{ for each } 1 \le j \le q.\] 

There are three ways a fixed $k \ge 1$ can relate to the sequence $\left(\mu_{j}\right)$: $\mu_{j-1} > k \ge \mu_{j}$ for some $j$, $1 \le k < \mu_{q}$, or $k \ge \mu_{0}$.  If $k \in \N$ with $\mu_{j-1} > k \ge \mu_{j}$ for some $j$, then $\sqrt{\fitt{\mu_{j}}{I}} = \sqrt{\fitt{k}{I}}$. To see this, for each $\fp \in V\!\left(\fitt{\mu_{j}}{I}\right)$ one has $\gens{R_{\fp}}{I_{\fp}} > \mu_{j}$; via property (iii) we see $\gens{R_{\fp}}{I_{\fp}} \ge \mu_{j-1} > k$, giving $\fp \in V\!\left(\fitt{k}{I}\right)$. On the other hand, if $1 \le k < \mu_{q}$, then $\sqrt{\fitt{k}{I}} = \sqrt{I}$ by Krull's Altitude Theorem and property (ii), giving $\hgt{\fitt{k}{I}} > k$. Finally, if $k \ge \mu_{0}$, then $\sqrt{\fitt{k}{I}} = R$ by property (i), so $\hgt{\fitt{k}{I}} > k$. Thus, we may restate the condition $G_{s}$ as \[\hgt{\fitt{\mu_{j}}{I}} \ge \min\!\left\{\mu_{j-1},s\right\} \text{ for each } 1 \le j \le q.\]  

We prove part (b) as it is the hardest case. We define the sequence $\mu_{j} = \frac{1}{n-j+1}\binom{n-j+1}{n-t}\binom{n-j+1}{n-t+1}$ for $0 \le j \le t-1$. One can show $\mu_{j}$ is strictly decreasing. Since $I$ is of generic symmetric height, $\mu_{t-1} = \hgt{I}$.  Moreover, $I$ can be generated by $\mu_{0}$ elements, since \cite[2.5]{Shafiei} proves $\detid{t}{I}$ may be generated by $\frac{1}{n+1}\binom{n+1}{t+1}\binom{n+1}{t}$ many elements.  Thus, it suffices to show $\left\{ \gens{R_{\fp}}{I_{\fp}} \:\middle\vert\:  \fp \in V\!\left(I\right)\right\} \subseteq \left\{ \mu_{j} \right\}$ and  $\sqrt{\detid{j}{A}} = \sqrt{\fitt{\mu_{j}}{I}}$.

Let $\,\mathfrak{p} \in V\!\left(I\right)$.  Suppose that in $R_{\mathfrak{p}}$, the matrix $A_{\mathfrak{p}}$ is equivalent to
\[ \left(\begin{array}{c|c|c}
I_{h \times h} &  0 & 0 \\
\hline
0 & J_{k \times k} & 0 \\
\hline
0 & 0 & C 
\end{array}\right), \]
where $I_{h \times h}$ is the $h \times h$ identity matrix, $J_{k \times k}$ is a $k \times k$ block diagonal matrix with each block of the form $\,\begin{pmatrix} a & u \\ u & b\end{pmatrix}\;$ so that $u$ is a unit and $a$ and $b$ are non-units, and $C$ is symmetric with no unit entries.   We refer to this as the \textit{standard form} of $A_{\fp}$. Note, for any $\fp \in V\!\left(I\right)$, the matrix $A_{\fp}$ can be placed into standard form without altering the ideal $\detid{t}{A_{\fp}}$. 

Observe that if $A_{\fp}$ is in standard form and $j = h+k$, then $\gens{R_{\fp}}{I_{\fp}} = \mu_{j}$. Indeed, note $I_{\mathfrak{p}} = \detid{t}{A_{\mathfrak{p}}} = \detid{t-j}{C}$, and then apply \cref{MinGens} to $\detid{t-j}{C}$. Therefore we have $\left\{ \gens{R_{\fp}}{I_{\fp}} \:\middle\vert\:  \fp \in V\!\left(I\right)\right\} \subseteq \left\{ \mu_{j} \right\}$.

Finally we show $\sqrt{\detid{j}{A}} = \sqrt{\fitt{\mu_{j}}{I}}$. Suppose $\fp \in V\!\left(I\right)$. Transform $A_{\fp}$ into standard form, and let $\ell = h+k$. Then, as above, $\mu_{\ell} = \gens{R_{\fp}}{I_{\fp}}$.
\begin{align*}
    \fp \notin V\!\left(\detid{j}{A}\right) &\iff A_{\fp} \text{ has a unit } j \times j \text{ minor}\\
     &\iff  j \le \ell \\
     &\iff \mu_{j} \ge \mu_{\ell} = \gens{R_{\fp}}{I_{\fp}}\\
     &\iff \fp \notin V\!\left(\fitt{\mu_{j}}{I}\right)
\end{align*}

Therefore, the result obtains as soon as we have proven \cref{MinGens}. The proofs for parts (a) and (c) are similar using the sequences $\mu_{j} = \binom{m-j}{m-t}\binom{n-j}{n-t}$ and $\mu_{j} = \binom{n-2j}{n-2t}$, respectively.
\end{proof}

\bigskip

\begin{lem}
\label{MinGens}
\hypertarget{MinGens}{}
Let $\left( R, \fm\right)$ be a Noetherian local ring and $1 \le t \le m \le n$.
\begin{enumerate}
    \item[$(\mathrm{a})$] Let $A$ be an $m \times n$ matrix with entries in $\fm$. If $I = \detid{t}{A}$ is of generic height, then $\gens{R}{I} = \binom{m}{t}\binom{n}{t}$.
    \item[$(\mathrm{b})$] Let $A$ be an $n \times n$ symmetric matrix with entries in $\fm$. If $I = \detid{t}{A}$ is of generic symmetric height, then $\gens{R}{I} = \frac{1}{n+1}\binom{n+1}{t+1}\binom{n+1}{t}$.
    \item[$(\mathrm{c})$] Let $A$ be an $n \times n$ alternating matrix with entries in $\fm$. If $I = \pfaffid{2t}{A}$ is of generic alternating height, then $\gens{R}{I} = \binom{n}{2t}$.
\end{enumerate}
\end{lem}
\begin{proof}
We begin with a reduction to the generic case. We prove this reduction for (b). The reduction is similar for (a) and (c). Let $X = \left(X_{ij}\right)$ be an $n\times n$ generic symmetric matrix over $R$, and define $B = R\!\left[X\right]$. Denote the unique maximal homogeneous $B$-ideal $\fm+\left(X\right)$ as $\fn$. Let $N$ be the $B$-ideal generated by the entries of the matrix $X-A$. We give $R$ the $B$-algebra structure induced by $R \cong B/N$. We note $R/\fm \cong B/\fn$ as $B$-algebras.  Since $N$ is generated by a regular sequence modulo $\detid{t}{X}$, we also have $\detid{t}{X} \otimes_{B} R \cong \detid{t}{A}$. Hence, 

\[ \detid{t}{A} \otimes_{R} R/\fm \cong \detid{t}{X} \otimes_{B} R  \otimes_{R} R/\fm \cong \detid{t}{X} \otimes_{B} B/\fn.\]

Therefore, $\detid{t}{A} \otimes_{R} R/\fm \cong \detid{t}{X} \otimes_{B} B/\fn$ as $B/\fn$-vector spaces. Hence, by Nakayama's lemma, $\gens{R}{\detid{t}{A}} = \gens{B}{\detid{t}{X}}$. 

For (a), we note $B$ is a free $R$-module with monomials in the entries of $X$ as a free basis. Thus, it suffices to show the distinct minors of $X$ are $R$-linearly independent. By the Laplace expansion of determinants, every monomial in the support of a $t \times t$ minor is a squarefree monomial of degree $t$. Given a monomial $\alpha$ in the support of a fixed  $t \times t$ minor $M$, if the indeterminate $X_{ij}$ divides $\alpha$, then any $t \times t$ submatrix of $X$ having determinant $M$ must use row $i$ and column $j$. Thus, since there are $t$ distinct indeterminates dividing $\alpha$, there is only one $t \times t$ submatrix of $X$ for which $\alpha$ is in the support of the determinant. Therefore, $\alpha$ determines the corresponding minor $M$. Hence, the distinct $t \times t$ minors of $X$ are $R$-linearly independent, and thus form a minimal generating set of $\detid{t}{X}$. There are $\binom{m}{t}\binom{n}{t}$ distinct $t \times t$ minors of $X$.

Similarly, for (c), we note $B$ is a free $R$-module with monomials in the entries of $X$ above the diagonal as a free basis. Thus, it suffices to show the distinct Pfaffians of $X$ are $R$-linearly independent. By the Laplace expansion of Pfaffians, every monomial in the support of a $2t \times 2t$ Pfaffian is a squarefree monomial of degree $t$. However, given a monomial $\alpha$ in the monomial support of a fixed  $2t \times 2t$ Pfaffian $P$, if the indeterminate $X_{ij}$ divides $\alpha$, then any principal submatrix of $X$ having Pfaffian $P$ must use row $i$ and column $j$. Thus, since there are $t$ distinct indeterminates dividing $\alpha$, there is only one principal submatrix of $X$ for which $\alpha$ is in the support of the Pfaffian. Therefore, $\alpha$ determines the corresponding Pfaffian $P$. Hence, the distinct $2t \times 2t$ Pfaffians of $X$ are $R$-linearly independent, and thus form a minimal generating set of $\pfaffid{2t}{X}$. There are $\binom{n}{2t}$ distinct principal submatrices $X$.

In case (b), we refer to \cite[2.5]{Shafiei} to obtain a count for the number of linearly independent $t \times t$ minors. In particular we obtain $\gens{B}{\detid{t}{X}} = \frac{1}{n+1}\binom{n+1}{t+1}\binom{n+1}{t}.$
\end{proof}

\bigskip

We will now classify the maximal $s$ for which $\detid{t}{X}$ satisfies $G_{s}$ where $X$ is a generic matrix.

\bigskip

\begin{cor}[Characterization of $G_{s}$ for generic matrices]
\label{GsGeneric}
\hypertarget{GsGeneric}{}Let $X$ be a generic $m \times n$ matrix, $1 \le t \le m \le n$, and $I = \detid{t}{X}$. 

Then $I$ satisfies $G_{\infty}$ if and only if one of the following conditions holds:

\begin{enumerate}
    \item[$(\mathrm{i})$] $t=1$.
    \item[$(\mathrm{ii})$] $t=m=n$.
    \item[$(\mathrm{iii})$] $m=n$ and $\,t=n-1$.
    \item[$(\mathrm{iv})$] $n=m+1$ and $\,t=m$.
    \item[$(\mathrm{v})$] $n=m+2$ and $\,t=m$.
    \item[$(\mathrm{vi})$] $m = 2$, $n = 5$, and $\,t = 2$.
\end{enumerate}

If $\;3 \le t = m$ and \,$n = m+3$, then the maximal $s$ for which $I$ satisfies $G_{s}$ is $s=18$.

In all other cases, the maximal $s$ for which $I$ satisfies $G_{s}$ is $s = \left(m-t+2\right)\left(n-t+2\right)$.
\end{cor}

\begin{proof}
From \hyperref[GsConditionsOrd]{\cref*{GsConditions}.a}, in order to satisfy $G_{\infty}$ we must have 
\begin{equation}(m-j+1)(n-j+1) = \hgt{\detid{j}{X}} \ge \binom{m-j+1}{m-t}\binom{n-j+1}{n-t} \tag{3.3.1}\label{3.3.1}\end{equation} 
for all $1 \le j \le t-1$.
Elementary computations reveal this condition is satisfied for cases (i)-(vi).

In order to obtain an upper bound on $s$ for the other cases, we show \labelcref{3.3.1} cannot be satisfied when substituting $j = t-1$, except in the cases (i)-(vi) where $G_{\infty}$ holds and for one exceptional case. 

When substituting $j = t-1$ into \labelcref{3.3.1}, we obtain the inequality $4 \ge \left(m-t+1\right)\left(n-t+1\right)$. The following table lists the cases where this inequality is satisfied. 

\[ \begin{array}{c|l c l}
(m-t+1)(n-t+1)  & & \mbox{Result}& \\
\hline \hline 1  &  t = m &\mbox{ and }& m = n\\
\hline 
    2 & t = m &\mbox{ and }& n = m + 1\\
    \hline
    3 & t = m &\mbox{ and }& n = m + 2\\
    \hline
    4 & t = m &\mbox{ and }& n = m + 3\\
      &       &\mbox{ or }&           \\
      & t = m - 1 &\mbox{ and }& n = m
\end{array} \]

The above table lists cases (i)-(vi) and the exceptional case $3 \le t = m$ and $n = m+3$. 

For the exceptional case $3 \le t = m$ and $n=m+3$, we substitute $j = t-2 = m-2$, to obtain $\binom{m-j+1}{m-t}\binom{n-j+1}{n-t} = 20 > 18 = \left(m-j+1\right)\left(n-j+1\right)$. Thus, the only way $G_{s}$ could be satisfied is if $18 = \hgt{\detid{m-2}{X}} \ge s$. Since $\hgt{\detid{j}{X}} \ge \hgt{\detid{m-2}{X}}$ for all $1 \le j \le m-2$, it follows that the maximal $s$ for which $G_{s}$ is satisfied is when $s = 18$.

Apart from the aforementioned cases, since $\hgt{\detid{t-1}{X}} < \binom{m-j+1}{m-t}\binom{n-j+1}{n-t}$ when $j = t-1$, the only way $G_{s}$ could be satisfied is if $(m-t+2)(n-t+2) =\hgt{\detid{t-1}{X}} \ge s$. Since $\hgt{\detid{j}{X}} \ge \hgt{\detid{t-1}{X}}$ for all $1 \le j \le t-2$, it follows that the maximal $s$ for which $G_{s}$ is satisfied is when $s = (m-t+2)(n-t+2)$. 
\end{proof}

\bigskip

\begin{cor}[Characterization of $G_{s}$ for generic symmetric matrices]
\label{GsSymmetric}
\hypertarget{GsSymmetric}{}Let $X$ be a generic symmetric $n \times n$ matrix, $1 \le t \le n$, and $I = \detid{t}{X}$. Then $I$ satisfies $G_{\infty}$ if and only if $t = 1$, $t = n$, or $t= n-1$. Otherwise, the maximal $s$ for which $I$ satisfies $G_{s}$ is $s = \binom{n-t+3}{2}$.
\end{cor}

\begin{proof}

From \hyperref[GsConditionsSym]{\cref{GsConditionsSym}.b}, in order for $G_{\infty}$ to be satisfied, we must have \begin{equation}\binom{n-j+2}{2} = \hgt{\detid{j}{X}} \ge \frac{1}{n-j+2}\binom{n-j+2}{n-t}\binom{n-j+2}{n-t+1} \tag{3.4.1}\label{3.4.1} \end{equation} for all  $1 \le j \le t-1$. Elementary computations reveal this condition is satisfied for cases $t=1, t=n$, and $t = n-1$.

In order to obtain an upper bound on $s$ for the other cases, we show \labelcref{3.4.1} cannot be satisfied when substituting $j = t-1$, except in the cases where $G_{\infty}$ holds. 

When substituting $j = t-1$ into \labelcref{3.4.1}, we obtain the inequality $6 \ge \left(n-t+2\right)\left(n-t+1\right)$. Restricting to $t \le n$, the inequality is satisfied for $n-1 \le t \le n$.

Apart from the aforementioned cases, since $\hgt{\detid{t-1}{X}} < \frac{1}{n-j+2}\binom{n-j+2}{n-t}\binom{n-j+2}{n-t+1}$ when $j = t-1$, the only way $G_{s}$ could be satisfied is if $\binom{n-t+3}{2} =\hgt{\detid{t-1}{X}} \ge s$. Since $\hgt{\detid{j}{X}} \ge \hgt{\detid{t-1}{X}}$ for all $1 \le j \le t-2$, it follows that the maximal $s$ for which $G_{s}$ is satisfied is when $s = \binom{n-t+3}{2}$. 
\end{proof}

\bigskip

\begin{cor}[Characterization of $G_{s}$ for generic alternating matrices]
\label{GsAlternating}
\hypertarget{GsAlternating}{}Let $X$ be a generic alternating $n \times n$ matrix, $2 \le 2t \le n$, and $I = \pfaffid{2t}{X}$. Then $I$ satisfies $G_{\infty}$ if and only if one of the following conditions holds:
\begin{enumerate}
\item[$(\mathrm{i})$] $2t = 2$,
\item[$(\mathrm{ii})$] $2t = n$,
\item[$(\mathrm{iii})$] $2t = n-2$, or
\item[$(\mathrm{iv})$] $2t = n-1$.
\end{enumerate}

Otherwise, the maximal $s$ for which $I$ satisfies $G_{s}$ is $s = \binom{n-2t+4}{2}$.
\end{cor}

\begin{proof}
From \hyperref[GsConditionsAlt]{\cref{GsConditionsAlt}.c}, in order for $G_{\infty}$ to be satisfied, we must have \begin{equation}\binom{n-2j+2}{2} = \hgt{\pfaffid{2j}{X}} \ge \binom{n-2j+2}{n-2t}\tag{3.5.1}\label{3.5.1}\end{equation} for all $1 \le j \le t-1$. Elementary computations reveal this condition is satisfied for cases (i)-(iv).

In order to obtain an upper bound on $s$ for the other cases, we show \labelcref{3.5.1} cannot be satisfied when substituting $j = t-1$, except in the cases where $G_{\infty}$ holds. 

When substituting $j = t-1$ into \labelcref{3.5.1} and restricting to $2 \le 2t \le n$, we obtain the inequality $n \ge 2t \ge n-2$. This happens only for cases (ii)-(iv) where $G_{\infty}$ is satisfied.

Apart from the aforementioned cases, since $\hgt{\detid{t-1}{X}} < \binom{n-2j+2}{n-2t}$ when $j = t-1$, the only way $G_{s}$ could be satisfied is if $\binom{n-2t+4}{2} =\hgt{\pfaffid{2t-2}{X}} \ge s$. Since $\hgt{\pfaffid{2j}{X}} \ge \hgt{\pfaffid{2t-2}{X}}$ for all $1 \le j \le t-2$, it follows that the maximal $s$ for which $G_{s}$ is satisfied is when $s = \binom{n-2t+4}{2}$. 
\end{proof}

\bigskip

\section{Approximation of Rees Rings and of Resolutions via Specialization}
\label{SectionSpecialization}
\hypertarget{SectionSpecialization}{}
Let $R$ be a Noetherian ring and let $X$ be a generic $m \times n$ matrix over $R$. For any $m \times n$ matrix $A$ whose entries are in $R$, there exists a surjective $R$-algebra homomorphism $R\!\left[X\right] \to R$ given by $X_{ij} \mapsto A_{ij}$. Moreover, for any $t$, the extension $\detid{t}{X}\!R$ via this $R$-algebra homomorphism is equal to $\detid{t}{A}$. In particular we have \[\dfrac{R\left[X\right]}{\left(\left\{X_{ij}-A_{ij} \mid 1 \le i \le m, 1 \le j \le n\right\}\right)} \cong R  \hspace{0.75cm} \text{ and } \hspace{0.75cm} \detid{t}{X} \otimes_{R[X]} R \cong \detid{t}{A}\] as $R$-modules.

More generally, if $B$ is a Noetherian ring, $J$ is a $B$-ideal, and $N = \left(Y_{1}, \ldots, Y_{D}\right)$ where $Y_{1}, \ldots, Y_{D}$ is a regular sequence on $B/J$, then $B/N$ is said to be a specialization of $B$ and $J\!\left(B/N\right)$ is a specialization of $J$. The primary goal of this section is to investigate how properties of the specialization of an ideal $J$ can be obtained from information about $J$ itself.

In particular, one would hope that understanding $\rees{\detid{t}{X}}$ would give enough information to understand $\rees{\detid{t}{A}}$ under suitable assumptions on $A$.

We now prove a technical lemma which will be critical in our results on specialization and degree bounds.

\bigskip

\begin{lem}
\label{TorsionKernelLemma}
\hypertarget{TorsionKernelLemma}{}Let $R$ be a Cohen-Macaulay ring, and suppose $M$ is a finitely generated $R$-module with finite projective dimension. Let $L$ be an $R$-module. Suppose $f \colon M \to L$ is an $R$-linear map and $\,\ker f = \tau_{R}\!\left(M\right)$.

\begin{enumerate}
    \item[$(\mathrm{i})$]\label{TorsionKernelLemmaFull}
\hypertarget{TorsionKernelLemmaFull}{} Then $f$ is injective if and only if $\; \pdim{R_{\fp}}{M_{\fp}} < \hgt{\fp}\,$ for all $\; \fp \in \spec{R}$ with $\hgt{\fp} > 0$.
    \item[$(\mathrm{ii})$]\label{TorsionKernelLemmaPunctured}
\hypertarget{TorsionKernelLemmaPunctured}{} Suppose $R$ is positively graded with $R_{0}$ local, $\dim R > 0$, $M$ and $L$ are graded $R$-modules, and $f$ is homogeneous. Denote the unique maximal homogeneous ideal of $R$ by $\,\fm$. Then $\,\ker f = \locoh{0}{\fm}{M}$ if and only if $\;\pdim{R_{\fp}}{M_{\fp}} < \hgt{\fp}$ for all $\;\fp$ in $\,\proj{R}$ with $\,\hgt{\fp} > 0$.
\end{enumerate}
\end{lem}

\begin{proof}
Let $K = \ker f$.

To prove (i), note $K = 0$ if and only if $K_{\fp} = 0$ for all $\fp \in \spec{R}$. Thus, we just have to show $f_{\fp}$ is injective locally if and only if $\pdim{R_{\fp}}{M_{\fp}} < \hgt{\fp}$ for all $\fp \in \spec{R}$ with $\hgt{\fp} > 0$. Since $K = \tau_{R}\!\left(M\right)$ and $R$ is Cohen-Macaulay, we see $f_{\fp}$ is injective for all $\fp \in \spec{R}$ if and only if $M_{\fp}$ is torsion-free for all $\fp \in \spec{R}$. Since $R$ is Cohen-Macaulay, $M_{\fp}$ is torsion-free for all $\fp \in \spec{R}$ if and only if $\depth{}{M_{\fp}} > 0$ for all $\fp \in \spec{R}$ with $\hgt{\fp} > 0$. Finally, by the Auslander-Buchsbaum Formula, $\depth{}{M_{\fp}} > 0$ for all $\fp \in \spec{R}$ with $\hgt{\fp} > 0$ if and only if $\pdim{R_{\fp}}{M_{\fp}} < \hgt{\fp}$ for all $\fp \in \spec{R}$ with $\hgt{\fp} > 0$.
    
We now prove (ii). Since $\locoh{0}{\fm}{M}$ consists of all elements of $M$ which are annihilated by a power of $\fm$, it follows that $\locoh{0}{\fm}{M} \subseteq \tau_{R}\!\left(M\right) = K \subseteq M$. Under these circumstances, $K = \locoh{0}{\fm}{M}$ if and only if $\supp{}{K} \subseteq \left\{ \fm \right\}$. Thus, since $f$ is homogeneous, we obtain $K = \locoh{0}{\fm}{M}$ if and only if $f_{\fp}$ is injective for all $\fp \in \proj{R}$. The string of equivalences in the proof of (i) is still valid when replacing $\spec{R}$ with $\proj{R}$. As such, we see $K = \locoh{0}{\fm}{M}$ if and only if $f_{\fp}$ is injective for all $\fp
\in \proj{R}$ if and only if $\pdim{R_{\fp}}{M_{\fp}} < \hgt{\fp}$ for all $\fp \in \proj{R}$ with $\hgt{\fp} > 0$.
\end{proof}

\bigskip

The next lemma is a main tool in our results which provides key insight into the specialization of powers of an ideal. While part (i) provides a resolution of the specialization of powers of an ideal, parts (ii) and (iii) are critical tools in providing degree bounds on $\mathcal{A}\!\left(I\right)$ via specialization.

\bigskip

\begin{lem}
\label{SpecializationLemma}
\hypertarget{SpecializationLemma}{}Let $B$ be a Cohen-Macaulay nonnegatively graded ring of dimension $d$ with $B_{0}$ local, and denote the unique maximal homogeneous ideal of $B$ by $\,\fm$. Let $J$ be a homogeneous $B$-ideal generated by forms of the same degree $q$. Suppose $Y_{1}, \ldots, Y_{D}$ is a homogeneous sequence in $B$ which is weakly regular on $B/J$ and on $B$. Let $N = \left(Y_{1},\ldots,Y_{D}\right)$, $R = B/N$, and $I = JR$.  For each positive integer $k$, let $\,\left(\comp{D^{k}},\comp{\varphi^{k}}\right)$ be a homogeneous finite free $B$-resolution of $\,J^{k}$ where each $\comp[i]{D^{k}}$ is finitely generated.
 
Suppose $\,\left\{K_{i}\right\}$ is a family of $B$-ideals so that
\[ K_{i} \subseteq \sqrt{I\left(\comp[i]{\varphi^{k}}\right)} \]
for all $i$ and for all $k$. 

\begin{enumerate}
    \item[$(\mathrm{i})$]\label{SpecializationLemmaRes}
\hypertarget{SpecializationLemmaRes}{} If $\hspace{1.5mm}\hgt{K_{i}R} \ge i$ for all $\,i$, then $\,\comp{D^{k}} \otimes_{B} R$ is a homogeneous free $R$-resolution of $\,J^{k} \otimes_{B} R$ for each $k$. 
    
    \item[$(\mathrm{ii})$]\label{SpecializationLemmaApproxRes}
\hypertarget{SpecializationLemmaApproxRes}{} If $\hspace{1.5mm}\hgt{K_{i}R} \ge \min\left\{i,d-1\right\}\,$ for all $\,i$, then $\,\comp{D^{k}} \otimes_{B} R$ is an approximate $R$-resolution of $\,J^{k} \otimes_{B} R$ for each $k$. 
    
    \item[$(\mathrm{iii})$]\label{SpecializationLemmaSeq}
\hypertarget{SpecializationLemmaSeq}{} Let $\,\psi\!_{k} \colon J^{k} \otimes_{B} R \to I^{k}$ be the natural surjection. For each $k$, there is a homogeneous exact sequence of $\,R$-modules
   \[ \A{k}{J} \otimes_{B} R \to \mcal{A}_{k}\!\left(I\right) \to \left(\ker \psi\!_{k}\right)\!\left(kq\right) \to 0. \]
   \end{enumerate}

\end{lem}

\begin{proof}
 Proof of (i). One sees $\comp{D^{k}} \otimes_{B} R$ is a finite complex of free $R$-modules with $\hmgy{0}{\comp{D^{k}} \otimes_{B} R} \cong J^{k} \otimes_{B} R$. For each $i$, $i \le \hgt{K_{i}R}  \le \hgt{\sqrt{I\left(\comp[i]{\varphi^{k}} \otimes_{B} R\right)}} = \hgt{I\left(\comp[i]{\varphi^{k}} \otimes_{B} R\right)} = \grade{I\left(\comp[i]{\varphi^{k}} \otimes_{B} R\right)}$, since $R$ is Cohen-Macaulay. Thus, $\comp{D^{k}} \otimes_{B} R$ is a resolution by the Buchsbaum-Eisenbud criterion \cite[1.4.13]{BH}. 
   
Proof of (ii). Recall the definition of an approximate resolution from \cref{ApproximateResolutionNotation}. It suffices to show $\dim \hmgy{j}{\comp{D^{k}} \otimes_{B} R} \le  j$ for all $1 \le j \le d-1$. To do this, fix $1 \le j \le d-1$. We will show $\comp{D^{k}} \otimes_{B} R_{\fp}$ is a resolution  for all $\mf{p} \in \spec{R}$ with $\hgt{\fp} \le d-j-1$. Hence, let $\mf{p} \in \spec{R}$ with $\hgt{\fp} \le d-j-1$. Then, for all $i$, $\hgt{K_{i}R_{\fp}} \ge \min\left\{ i, d-1 \right\} \ge \min\left\{ i, d-j \right\}$. In either case, $\hgt{K_{i}R_{\fp}} \ge i$ given our assumption on $\hgt{\fp}$. Hence, by part (i), $\hmgy{i}{\comp{D^{k}} \otimes_{B} R}_{\mf{p}} = 0$ for $\hgt{\fp} \le d-j-1$. Thus, $\dim \hmgy{j}{\comp{D^{k}} \otimes_{B} R} \le j$ for all $1 \le j \le d-1$.

We now prove (iii). First we prove $\psi_{1}$ is an isomorphism. To do this, note $\ker \psi_{1} = \left(N\cap J\right)/NJ = \Tor{1}{B}{B/N}{B/J} = 0$ since $Y_{1},\ldots,Y_{D}$ is weakly regular on $B$ and on $B/J$. In particular, $\sym{k}{J} \otimes_{B} R \cong \sym{k}{I}$ for all $k$.

Thus, for each $k$, we have the following commutative diagram with exact rows.
	
\[ \begin{tikzcd}
 & \A{k}{J} \otimes_{B} R \ar[d] \ar[r] & \sym{k}{I} \ar[d, equals] \ar[r] & J^{k} \otimes_{B} R \ar[r] \ar[d, two heads, "\psi_{k}"] & 0 \\
0 \ar[r] & \A{k}{I} \ar[r] & \sym{k}{I} \ar[r]  & I^{k} \ar[r] & 0
\end{tikzcd} \]

Therefore, by the Snake Lemma, the sequence
\[ \A{k}{J} \otimes_{B} R \to \A{k}{I} \to \ker \psi\!_{k} \to 0 \]
is exact.

Keeping in mind the gradings on the symmetric algebra and the Rees ring as in \cref{ReesNotation} and \cref{SymmetricNotation}, the sequence is homogeneous provided we use $\left(\ker \psi\!_{k}\right)\left(kq\right)$.
\end{proof}

\bigskip

We will mainly use this lemma in the case described at the beginning of this section. The family of ideals $K_{i}$ will be a subcollection of the family of ideals $\sqrt{\detid{j}{X}}$, which will allow us to use the heights of the ideals $\detid{j}{A}$ to  control  $\A{}{I}$. This will be shown in \cref{ContainmentLemma}; however, we must first prove some well-known technical lemmas (see, for instance, \cite{CN} for \cref{CNAnalyticSpread} and \cite[2.4]{BV} for \cref{ReductionLemma}).

In order to prove the aforementioned technical lemmas, we need information about minimal free resolutions of $\detid{j}{X}^{k}$ and $\pfaffid{2j}{X}^{k}$. Of course the best scenario is when explicit minimal free resolutions of $\detid{j}{X}^{k}$ or $\pfaffid{2j}{X}^{k}$ are known for all $k$. There are two cases where we actually have this information, which we state as the following two propositions.

\bigskip

\begin{prop}[Akin-Buchsbaum-Weyman {\cite[5.4]{ABW}}]
\label[proposition]{ResolutionPowersHilbertBurch}
\hypertarget{ResolutionPowersHilbertBurch}{}
Let $R$ be a nonnegatively graded Noetherian ring with $R_{0}$ local, $m$ and $n$ integers satisfying $1 \le m \le n$, $X$ an $m \times n$ generic matrix over $R$, $S = \poly{R}{X}$, and $J = \detid{m}{X}$. For any positive integer $k$, let $\;\compp{F^{k}}$ be the minimal homogeneous free $S$-resolution of $\,J^{k}\!\paren{kt}$. Then $\;\compp{F^{k}}$ is a linear resolution of length $\min\!\brces{k,m}\paren{n-m}$.
\end{prop}

\bigskip

This proposition is implicit from the proof of \cite[5.4]{ABW} but is not directly stated. A direct statement and justification for the length of the resolution can be found in \cite[3.1]{BCV}, and a direct statement for the linearity of the resolution can be found in \cite[3.6]{BCV}.

\bigskip

\begin{prop}[Kustin-Ulrich {\cite[2.7, 4.7, 4.13.b]{KU}}]
\label[proposition]{KUComplexes}
\hypertarget{KUComplexes}{}
Let $K$ be a field, $n$ an odd integer with $3 \le n$, $X$ an $n \times n$ generic alternating matrix over $R$, $S = \poly{K}{X}$, and $J = \pfaffid{n-1}{X}$. The minimal free resolution of $\,J^{k}\!\paren{k\paren{n-1}/2}$ is a complex $\compp{\mathcal{D}^{k}}$ with
\[ \compp[i]{\mathcal{D}^{k}} = \begin{cases}
S\!\paren{-i}^{\beta_{i}^{k}}, & \text{ if } i \le \min\!\brces{k,n-1} \\
S\!\paren{-\paren{i-1}-\frac{1}{2}\paren{n-i+1}}, & \text{ if } i = k+1 \le n-1 \text{ and } k \text{ is odd}\\
0, & \text{ if } i = k+1  \text{ and } k \text{ is even}\\
0, & \text{ if } i \ge  \min\!\brces{k+2,n}
\end{cases}\]
for some nonzero $\beta_{i}$.
\end{prop}

\bigskip

Although proven in \cite[2.7, 4.7, 4.13.b]{KU}, the details of the above formulation are given in \cite[Proof of 6.1.b]{KPU}, specifically in equations (6.1.4)-(6.1.7).

For other types of determinantal and Pfaffian ideals, we do not know explicit resolutions of $\detid{j}{X}^{k}$ or $\pfaffid{2j}{X}^{k}$ for all $k$. Nevertheless, we are able to compute properties of such ideals as in the following lemma. 

\bigskip

\begin{lem}
\label{CNAnalyticSpread}
\hypertarget{CNAnalyticSpread}{} Let $R$ be a polynomial ring in $d$ indeterminates over the field $K$. 
\begin{enumerate}
    \item[$(\mathrm{i})$] Let $X$ be an $m \times n$ generic matrix over $R$ with $\;1 \le t = m \le n$, $B=R\!\left[X\right]$, and $J=\detid{m}{X}$.  Then $\max_{k}\!\left\{\pdim{B}{J^{k}}\right\}  = m\left(n-m\right)$.
    
    \item[$(\mathrm{ii})$] Let $X$ be an $m \times n$ generic matrix over $R$ with $\;1 \le t < m \le n$, $B=R\!\left[X\right]$, and $J=\detid{t}{X}$. Then $\max_{k}\!\left\{\pdim{B}{J^{k}}\right\}  = mn-1$.
    
    \item[$(\mathrm{iii})$] Let $X$ be an $n \times n$ generic symmetric matrix over $R$, $B=R\!\left[X\right]$, $\;1 \le t < n$, and $J=\detid{t}{X}$. Then $\max_{k}\!\left\{\pdim{B}{J^{k}}\right\}  = \binom{n+1}{2}-1$.
    
     \item[$(\mathrm{iv})$] Let $X$ be an $n \times n$ generic alternating matrix over $R$, $\;2 \le 2t = n-1$,  $B=R\!\left[X\right]$, and $J=\pfaffid{n-1}{X}$. Then $\max_{k}\!\left\{\pdim{B}{J^{k}}\right\}  =  n-1$.
     
     \item[$(\mathrm{v})$] Let $X$ be an $n \times n$ generic alternating matrix over $R$, $B=R\!\left[X\right]$, $\;2 \le 2t < n-1$, and $J= \pfaffid{2t}{X}$.  Then $\max_{k}\!\left\{\pdim{B}{J^{k}}\right\}  = \binom{n}{2}-1$.
\end{enumerate}
\end{lem}
\begin{proof}
Since $B = R\!\left[X\right]$ is a polynomial ring over $K\!\left[X\right]$ and since the generators of $J^{k}$ are elements of $K\!\left[X\right]$, $\pdim{B}{J^{k}} = \pdim{\polySmall{K}{X}}{J^{k}}$, so we may reduce to the case $R = K$. 

To prove (i) and (iv), we use the known free resolutions given in \cref{ResolutionPowersHilbertBurch} and \cref{KUComplexes}, respectively.

For (ii), (iii), and (v), by Hilbert's Syzygy Theorem, we have $\pdim{B}{\;B/J^{k}} \le \dim B$ for all $k$; thus, $\max_{k}\!\left\{\pdim{B}{\;J^{k}}\right\} \le \dim B -1$. On the other hand, we can use Burch's inequality (see \cref{AnalyticSpreadNotation}) and the Auslander-Buchsbaum Formula to obtain $\aspread{J}-1 \leq \max_{k}\!\left\{\pdim{B}{J^k}\right\}$. Thus, after we have shown $\aspread{J} = \dim B$, the results obtain.

We repurpose the argument found in \cite{CN} to compute the analytic spread in case (ii). For cases (iii) and (v) a similar argument will hold with some small adjustments. Let $\eltlist{\Delta}{1}{s}$ be the $t \times t$ minors of $X$, and let $F := K\!\left( \eltlist{\Delta}{1}{s} \right)$ be the fraction field of $\fiber{J}$. We first perform the computation in the special case where $n = m = t+1$.

It follows from $\adj X \cdot X = \left( \det X \right) \cdot I_{m \times m}$ that $\left( \det X \right)^{m-1} = \det{ \left( \adj X \right)} \in F$. Hence $F \subseteq F\!\left(\det X \right)$ is an algebraic extension.

Now use $\adj X \cdot X = \left( \det X \right) \cdot I_{m \times m}$ with Cramer's rule to see $X_{ij} \in F\!\left( \det X \right)$ for all $i$ and $j$. Specifically, we have \[ X_{ij} = \frac{\det{\left(Z_{ij}\right)}}{ \det{\left(\adj X\right)}}    \in F(\det X) \tag{4.3.1}\label{4.3.1}\] where $Z_{ij}$ is the matrix obtained from $\adj X$ by replacing column $j$ with column $i$ from $\left( \det X \right) \cdot I_{m \times m}$. Thus we see $X_{ij}$ is algebraic over $F$ for all $i$ and $j$.

In the general case we instead consider a submatrix $X'$ of $X$ with size $(t+1) \times (t+1)$. The special case above shows the entries of $X'$ are algebraic over the subfield generated by the $t \times t$ minors of $X'$. Finally we allow $X'$ to vary over all of the $(t+1) \times (t+1)$ submatrices of $X$, whence $X_{ij}$ is algebraic over $F = \quot{\fiber{J}}$ for all $i$ and $j$. Hence, \[\aspread{J} = \dim \fiber{J} = \trdeg_{K} \quot{\fiber{J}} = \dim B.\]

For case (v) we note some adjustments necessary to proceed with the same argument as above. When considering a submatrix one only uses principal submatrices of size $2t+2$, and one uses the Pfaffian adjoint as discussed in \cref{AdjNotation} in place of the classical adjoint. 
\end{proof}

\bigskip

After localization, a generic matrix remains generic though of smaller dimensions. This well-known reduction is critical to our methods, so we state and prove it here.

\bigskip

\begin{lem}
\label{ReductionLemma}
\hypertarget{ReductionLemma}{}
Let $B$ be a Noetherian ring, $1 \le j \le t \le m \le n$, $X = \left(\begin{array}{c|c}
U & V \\
\hline
W & Z
\end{array}\right)$ be an $m \times n$ generic matrix over $B$ with $U$ a $j \times j$ matrix, and $Y = \left(Y_{rs}\right)$ be an $\left(m-j\right) \times \left(n-j\right)$ generic matrix over $B$ Write $\Delta = \det U$ and $C = B\!\left[U,V,W, \Delta^{-1}\right]$. There is an $C$-algebra isomorphism
\[ \varphi \colon C\!\left[Z\right] \to C\!\left[Y\right] \]
given by
\[ Z_{rs} \mapsto \left(Y + WU^{-1}V\right)_{rs}.  \]

Moreover, the extension of $\detid{t}{X}C\!\left[Z\right]$ along $\varphi$ is $\detid{t-j}{Y}C\!\left[Y\right]$.
\end{lem}

\begin{proof} First notice $\detid{t}{X} C\!\left[Z\right] = \detid{t-j}{\wtld X} C\!\left[Z\right]$, where $\wtld X = Z-WU^{-1}V$ is an $(m-j)\times(n-j)$ matrix with entries in $C\!\left[Z\right]$. This is shown by the following block matrix computation performed using elementary row and column operations defined over $C$:

\[ \left(\begin{array}{c|c}
I & 0\\
\hline
-WU^{-1} & I
\end{array}\right)
\left(\begin{array}{c|c}
U & V\\
\hline
W & Z
\end{array}\right)
\left(\begin{array}{c|c}
I & -U^{-1}V\\
\hline
0 & I
\end{array}\right) =  \left(\begin{array}{c|c}
U & 0\\
\hline
0 & Z-WU^{-1}V
\end{array}\right).\]

Therefore, the image of $\detid{t}{X}C\!\left[Z\right]$ via $\varphi$ is $\detid{t-j}{Y}C\!\left[Y\right]$. 

To see $\varphi$ is an isomorphism simply consider the $C$-algebra homomorphism $\psi \colon C\!\left[Y\right] \to C\!\left[Z\right]$ defined by $Y_{rs} \mapsto \left(Z - WU^{-1}V\right)_{rs}$. One can see these maps are inverse to each other.
\end{proof}

\bigskip

We note the above proof also holds assuming $X$ and $Y$ are generic symmetric. In this case, $U$ and $Z$ are symmetric and $W = V^{T}$; hence, the maps $\varphi$ and $\psi$ are well-defined and send corresponding entries of $\widetilde{X}$ and $Y$ to each other.  The above proof also holds assuming $X$ and $Y$ are generic alternating and replacing ideals of minors with ideals of Pfaffians, provided we assume $t$ and $j$ are even. Indeed, in this case, $U$ and $Z$ are alternating and $W = -V^{T}$; hence, the maps $\varphi$ and $\psi$ are well-defined and send corresponding entries of $\widetilde{X}$ and $Y$ to each other.

\bigskip

\begin{lem}
\label{ContainmentLemma}
\hypertarget{ContainmentLemma}{}
Let $K$ be a field, and suppose $R$ is a polynomial ring in $d$ indeterminates over $K$.
\begin{enumerate}
    \item[$(\mathrm{i})$] Let $X$ be an $m \times n$ generic matrix over $R$ with $1 \le t = m \le n$, $B=R\!\left[X\right]$, and $J=\detid{m}{X}$. Further, for each $k$, let $\left(\comp[\bullet]{D^{k}},\comp{\varphi^{k}}\right)$ be a finite free $B$-resolution of $J^{k}$ where each $\comp[i]{D^{k}}$ is finitely generated. For each $1 \le j \le m-1$ and for all $k$, \[\sqrt{\detid{j}{X}} \subseteq \sqrt{\detid{}{\comp[i]{\varphi^{k}}}} \text{ whenever } i \ge \left(m-j\right)\left(n-m\right)+1.\]
    
    \item[$(\mathrm{ii})$] Let $X$ be an $m \times n$ generic matrix over $R$ with $1 \le t < m \le n$, $B=R\!\left[X\right]$, and $J=\detid{t}{X}$. Further, for each $k$, let $\left(\comp[\bullet]{D^{k}},\comp{\varphi^{k}}\right)$ be a finite free $B$-resolution of $J^{k}$ where each $\comp[i]{D^{k}}$ is finitely generated. For each $1 \le j \le t-1$ and for all $k$, \[\sqrt{\detid{j}{X}} \subseteq \sqrt{\detid{}{\comp[i]{\varphi^{k}}}} \text{ whenever } i \ge \left(m-j\right)\left(n-j\right).\]
    
    \item[$(\mathrm{iii})$] Let $X$ be an $n \times n$ generic symmetric matrix over $R$, $B=R\!\left[X\right]$, $1 \le t \le n$, and $J=\detid{t}{X}$. Further, for each $k$, let $\left(\comp[\bullet]{D^{k}},\comp{\varphi^{k}}\right)$ be a finite free $B$-resolution of $J^{k}$ where each $\comp[i]{D^{k}}$ is finitely generated. For each $1 \le j \le t-1$ and for all $k$,\[\sqrt{\detid{j}{X}} \subseteq \sqrt{\detid{}{\comp[i]{\varphi^{k}}}} \text{ whenever } i \ge \binom{n-j+1}{2}.\]
    
     \item[$(\mathrm{iv})$] Let $X$ be an $n \times n$ generic alternating matrix over $R$, $2 \le 2t = n-1$,  $B=R\!\left[X\right]$, and $J=\pfaffid{n-1}{X}$. Further, for each $k$, let $\left(\comp[\bullet]{D^{k}},\comp{\varphi^{k}}\right)$ be a finite free $B$-resolution of $J^{k}$ where each $\comp[i]{D^{k}}$ is finitely generated. For each $2 \le 2j \le n-3$ and for all $k$, \[\sqrt{\pfaffid{2j}{X}} \subseteq \sqrt{\detid{}{\comp[i]{\varphi^{k}}}} \text{ whenever } i \ge n-2j.\]
     
     \item[$(\mathrm{v})$] Let $X$ be an $n \times n$ generic alternating matrix over $R$, $B=R\!\left[X\right]$, $2 \le 2t < n-1$, and $J=\pfaffid{2t}{X}$. Further, for each $k$, let $\left(\comp[\bullet]{D^{k}},\comp{\varphi^{k}}\right)$ be a finite free $B$-resolution of $J^{k}$ where each $\comp[i]{D^{k}}$ is finitely generated. For each $1 \le j \le t-1$ and for all $k$, \[\sqrt{\pfaffid{2j}{X}} \subseteq \sqrt{\detid{}{\comp[i]{\varphi^{k}}}} \text{ whenever } i \ge \binom{n-2j}{2}.\]
\end{enumerate}
\end{lem}

\bigskip

\begin{proof} 

We prove parts (i) and (ii). Parts (iii)-(v) are similar.

Fix $1 \le j \le t-1$. We will show that if $\fp \in \spec{B}\setminus V\!\left(\detid{j}{X}\right)$, then $\fp \in \spec{B}\setminus V\!\left(\detid{}{\comp[i]{\varphi^{k}}}\right)$ for the appropriate values of $i$, and for all $k$. Let $\fp \in \spec{B}\setminus V\!\left(\detid{j}{X}\right)$. It suffices to show $\pdim{B_{\fp}}J^{k}_{\fp} < i$ for all $i$ in the appropriate range and for all $k$. We use \cref{ReductionLemma}, where the notation is consistent with our current setting. We see $B_{\fp}$ is a localization of $C\left[Z\right]$ from \cref{ReductionLemma}. Therefore \[\pdim{B_{\fp}} J^{k}_{\fp} \le \pdim{C\left[Z\right]} \detid{t}{X}^{k} = \pdim{C\left[Y\right]} \detid{t-j}{Y}^{k} = \pdim{R\left[U,\Delta^{-1},Y\right]}\detid{t-j}{Y}^{k} \]\[= \pdim{R\left[U,Y\right]_{\Delta}} \detid{t-j}{Y}^{k} \le \pdim{R\left[U,Y\right]}\detid{t-j}{Y}^{k} = \pdim{K\left[Y\right]}\detid{t-j}{Y}^{k} \] for all $k$. The above string of equalities and inequalities follows from two main facts. First, if $A$ is a Noetherian ring and $\mathfrak{a}$ is an $A$-ideal, then $\pdim{A_{\fp}}{\mathfrak{a}_{\fp}} \le \pdim{A}{\mathfrak{a}}$ for all $\fp \in \spec{A}$. Second, if $A$ is a Noetherian ring and $\mathfrak{a}$ is an $\poly{A}{x}$-ideal whose generators are elements of $A$, then $\pdim{\polySmall{A}{x}}{\mathfrak{a}} = \pdim{A}{\mathfrak{a}\cap A}$. Now apply \cref{CNAnalyticSpread} to obtain an upper bound on $\pdim{K\left[Y\right]}\detid{t-j}{Y}^{k}$, and hence, on $\pdim{B_{\fp}} J^{k}_{\fp}$.
\end{proof}

\bigskip

The next proposition addresses cases where the Rees ring specializes. Much of the proposition is a direct consequence of \cite[1.1]{EH}, and (i) is even given in \cite[3.5]{EH}. A crucial hypothesis in \cite[1.1]{EH} is that $\rees{J}$ is Cohen-Macaulay. For most of the cases below, it is known that $\rees{J}$ is Cohen-Macaulay, but we do not know whether $\rees{J}$ is Cohen-Macaulay in case (ii) when $0 < \CHAR K \le \min\brces{t,m-t}$, in case (iii), or in case (v) when $0 < \CHAR K \le \min\brces{2t,n-2t}$. We note that our method to conclude the Rees ring specializes does not rely on $\rees{J}$ being Cohen-Macaulay. Hence, the lemma below provides new results in the cases listed above. It also provides a new method even in the cases which were previously known. Additionally, we will make use of the following result to guarantee certain determinantal and Pfaffian ideals are of linear type or of fiber type.

\bigskip

\begin{prop}
\label{SpecializationRees}
\hypertarget{SpecializationRees}{}
Let $K$ be a field. Suppose $R = K\!\left[x_{1},\ldots,x_{d}\right]$ is a polynomial ring in $d$ indeterminates over $K$.
\begin{enumerate}
    \item[$(\mathrm{i})$] Let $A$ be an $m \times n$ matrix over $R$ with $1 \le t = m \le n$, $I = \detid{m}{A}$ be of generic height, $X$ be a generic $m \times n$ matrix over $R$, $B=R\!\left[X\right]$, and $J=\detid{m}{X}$. If 
    \[ \hgt{\detid{j}{A}} \ge \left(m-j+1\right)\left(n-m\right)+1 \hspace{1mm}\text{ for all }\hspace{1mm} 1 \le j \le m-1, \]
    then $\;\rees{J} \otimes_{B} R \cong \rees{I}$ via the natural map. Moreover, $\rees{I}$ is Cohen-Macaulay.
    
    \item[$(\mathrm{ii})$] Let $A$ be an $m \times n$ matrix over $R$, $1 \le t < m \le n$, $I = \detid{t}{A}$ be of generic height, $X$ be a generic $m \times n$ matrix over $R$,  $B=R\!\left[X\right]$, and $J=\detid{t}{X}$. 
    If 
    \[ \hgt{\detid{j}{A}} \ge \left(m-j+1\right)\left(n-j+1\right) \hspace{1mm}\text{ for all }\hspace{1mm} 1 \le j \le t-1, \]
    then $\;\rees{J} \otimes_{B} R \cong \rees{I}$ via the natural map. 
    
    If, in addition, $\;\chr K = 0$ or $\;\chr K > \min\!\left\{t, m-t \right\}$, then $\;\rees{I}$ is Cohen-Macaulay.

    \item[$(\mathrm{iii})$] Let $A$ be an $n \times n$  symmetric matrix over $R$, $1 \le t \le n$, $I = \detid{t}{A}$ be of generic symmetric height, $X$ be an $n \times n$ generic symmetric matrix over $R$, $B=R\!\left[X\right]$, and $J=\detid{t}{X}$. 
    If 
    \[ \hgt{\detid{j}{A}} \ge \binom{n-j+2}{2} \hspace{1mm}\text{ for all }\hspace{1mm} 1 \le j \le t-1, \]
    then $\;\rees{J} \otimes_{B} R \cong \rees{I}$ via the natural map.

     \item[$(\mathrm{iv})$] Let $A$ be an $n \times n$  alternating matrix over $R$ with $n$ odd, $2 \le 2t = n-1$, $I = \pfaffid{n-1}{A}$ be of generic alternating height, $X$ be an $n \times n$ generic alternating matrix over $R$, $B=R\!\left[X\right]$, and $J=\pfaffid{n-1}{X}$. 
     If 
    \[ \hgt{\pfaffid{2j}{A}} \ge n-2j+2 \hspace{1mm}\text{ for all }\hspace{1mm} 2 \le 2j \le n-3, \]
    then $\;\rees{J} \otimes_{B} R \cong \rees{I}$ via the natural map. Moreover, $\,\rees{I}$ is Cohen-Macaulay.
     
     \item[$(\mathrm{v})$] Let $A$ be an $n \times n$ alternating matrix over $R$, $2 \le 2t < n-1$, $I = \pfaffid{2t}{A}$ be of generic alternating height, $X$ be an $n \times n$ generic alternating matrix over $R$, $B=R\!\left[X\right]$, and $J=\pfaffid{2t}{X}$. 
     If 
    \[ \hgt{\pfaffid{2j}{A}} \ge \binom{n-2j+2}{2} \hspace{1mm}\text{ for all }\hspace{1mm} 2 \le 2j \le 2t-2, \]
    then $\;\rees{J} \otimes_{B} R \cong \rees{I}$ via the natural map. 
    
    If, in addition, $\,\chr K = 0$ or $\,\chr K > \min\!\left\{2t, n-2t \right\}$, then $\;\rees{I}$ is Cohen-Macaulay.
\end{enumerate}
\end{prop}

\begin{proof}
Let $N$ be the $B$-ideal generated by the entries of $X-A$, and give $R$ the $B$-algebra structure induced by the isomorphism $R \cong B/N$. Let $\psi_{k} \colon J^{k} \otimes_{B} R \to I^{k}$ be the natural surjection. Then $\bigoplus_{k=0}^{\infty} \psi_{k} \colon \rees{J} \otimes_{B} R \to \rees{I}$ is the natural surjection. Hence, to show the Rees ring specializes, it suffices to prove $\ker \psi_{k} = 0$ for all $k$.

We see $\ker \psi_{k} = \tau_{R}\!\left(J^{k} \otimes_{B} R\right)$. Indeed,
\[  \left(J^{k} \otimes_{B} R\right) \otimes_{R} \quot{R} \cong J^{k} \otimes_{B} R \otimes_{R} \frac{B_{N}}{NB_{N}} \cong  J^{k} \otimes_{B} \frac{B_{N}}{NB_{N}} \cong \frac{B_{N}}{NB_{N}} \cong \quot{R} \]
since $J$ is not contained in $N$. In addition, $I^{k} \otimes_{R} \quot{R} \cong \quot{R}$. Hence, tensoring $\psi\!_{k}$ with $\quot{R}$ is an isomorphism, giving that $\ker \psi\!_{k}$ is a torsion $R$-module. Clearly, the $R$-torsion of $J^{k} \otimes_{B} R$ is contained in $\ker \psi\!_{k}$ since $I^{k}$ is $R$-torsion-free.

Therefore, by \hyperref[TorsionKernelLemmaFull]{\cref{TorsionKernelLemmaFull}.i}, it suffices to show $\pdim{R_{\fp}}\!\left(J^{k} \otimes_{B} R\right)_{\fp} < \hgt{\fp}$ for all $\fp \in \spec{R}$ with $\hgt{\fp} > 0$. 

In order to show this, we first prove resolutions of $J^{k}$ specialize to resolutions of $J^{k} \otimes_{B} R$. Specifically, for each $k$, let $\left(\comp[\bullet]{D^{k}},\comp{\varphi^{k}}\right)$ be a finite free $B$-resolution of $J^{k}$ where each $\comp[i]{D^{k}}$ is finitely generated. We prove $\comp{D^{k}} \otimes_{B} R$ is a free $R$-resolution. By \hyperref[SpecializationLemmaRes]{\cref{SpecializationLemmaRes}.i}, it suffices to find a family $\left\{K_{i}\right\}$ of $B$-ideals with $\hgt{K_{i}R} \ge i$ for each $i$ and with $K_{i} \subseteq \sqrt{\detid{}{\comp[i]{\varphi^{k}}}}$ for each $i$ and for all $k$. We handle all cases simultaneously by establishing the following notation:

\begin{enumerate}
    \item[$(\mathrm{i})$] Let $I_{j} = \detid{j}{X}$ and $\sigma\!\left(j\right) = \left(m-j\right)\left(n-m\right)+1$.
    
    \item[$(\mathrm{ii})$] Let $I_{j} = \detid{j}{X}$ and $\sigma\!\left(j\right) = \left(m-j\right)\left(n-j\right)$.
    
    \item[$(\mathrm{iii})$] Let $I_{j} = \detid{j}{X}$ and $\sigma\!\left(j\right) = \binom{n-j+1}{2}$.
    
    \item[$(\mathrm{iv})$] Let $I_{j} = \pfaffid{2j}{X}$ and $\sigma\!\left(j\right) = n-2j$.
    
    \item[$(\mathrm{v})$] Let $I_{j} = \pfaffid{2j}{X}$ and $\sigma\!\left(j\right) = \binom{n-2j}{2}$.
\end{enumerate}

For each $1 \le i \le \hgt{I}-1$, we set $K_{i} = \sqrt{J}$.  For each $\hgt{I} \le i \le \max_{k}\!\left\{\pdim{}{J^{k}}\right\}$, we let $j_{0}$ be the smallest $j$ satisfying $i \ge \sigma\!\left(j\right)$, and set $K_{i} = \sqrt{I_{j_{0}}}$. 

We now show that for each $i$ satisfying $\hgt{I} \le i \le \max_{k}\!\left\{\pdim{}{J^{k}}\right\}$, we have $1 \le j_{0} \le t-1$, which gives us control over $\hgt{K_{i}R}$ by assumption. Since $I$ is of generic (symmetric, alternating) height, $\sigma\!\left(t-1\right) = \hgt{I} \le i$. Thus, by the definition of $j_{0}$, we have $j_{0} \le t-1$. Moreover, by \cref{CNAnalyticSpread}, we see $\sigma\!\left(0\right) =  \max_{k}\!\left\{\pdim{}{J^{k}}\right\}+1 > i$. Thus, by the definition of $j_{0}$, we have $1 \le j_{0}$.

We now proceed to demonstrate the required height and containment conditions for the family $\{K_i\}$.

Fix $i$ with $1 \le i \le \hgt{I}-1$.  Note that we can extend $\left(\comp{D^{k}},\comp{\varphi^{k}}\right)$ to a resolution of $B/J^{k}$ by taking $0 \to \comp{D^{k}} \overset{\comp[0]{\varphi^{k}}}{\to} B$.  Then, since $\ann{}{B/J^{k}} \ne 0$ and $B$ is Cohen-Macaulay, we have $\sqrt{J^{k}} = \sqrt{\detid{}{\comp[0]{\varphi^{k}}}} = \cdots = \sqrt{\detid{}{\comp[\hgt{J^{k}-1}]{\varphi^{k}}}}$. Since $I$ is of generic height, $\hgt{J^{k}} = \hgt{J} = \hgt{I}$.  In particular, we have $K_{i} = \sqrt{J} = \sqrt{J^{k}} \subseteq \sqrt{\detid{}{\comp[i]{\varphi^{k}}}}$, and $\hgt{K_{i}R} > i$. 

Now, fix $i$ with $\hgt{I} \le i \le \max_{k}\!\left\{\pdim{}{J^{k}}\right\}$, and recall $j_{0}$ satisfies $\sigma\!\left(j_{0}-1\right) > i \ge \sigma\!\left(j_{0}\right)$ and $1 \le j_{0} \le t-1$. Since $i \ge \sigma\!\left(j_{0}\right)$, by \cref{ContainmentLemma}, we have $K_{i} = \sqrt{I_{j_{0}}} \subseteq \sqrt{\detid{}{\comp[i]{\varphi^{k}}}}$.  Moreover, by assumption, $\hgt{K_{i}R} = \hgt{I_{j_{0}}}R = \hgt{\detid{j_{0}}{A}} \ge \sigma\!\left(j_{0}-1\right) > i$. 

Therefore, $\comp[\bullet]{D^{k}} \otimes_{B} R$ is a free $R$-resolution of $J^{k} \otimes_{B} R$. To finish the proof of specialization, we note that we have actually proven more than $\comp[\bullet]{D^{k}} \otimes_{B} R$ being a free $R$-resolution of $J^{k} \otimes_{B} R$.  Indeed, we only needed $\hgt{K_{i}R} \ge i$ for all $i$, yet we proved a strict inequality for all $i$. Hence, for any positive integer $i$, we get $\hgt{\detid{}{\comp[i]{\varphi^{k}}\otimes_{B} R}} > i$. Therefore, we conclude $\pdim{R_{\fp}}\!\left(J^{k} \otimes_{B} R\right)_{\fp} < \hgt{\fp}$ for all $\fp \in \spec{R}$ with $\hgt{\fp} > 0$, finishing the proof of specialization.

All that remains to prove are the statements about $\rees{I}$ being Cohen-Macaulay. By \cite[1.1]{EH}, $\rees{J} \otimes_{B} R$ is Cohen-Macaulay when $\rees{J}$ is Cohen-Macaulay. For (i), $\rees{J}$ is Cohen-Macaulay \cite[3.5]{EH}.  For (ii), $\rees{J}$ is Cohen-Macaulay if $\chr K = 0$ or $\chr K > \min\!\left\{t,m-t\right\}$ by \cite[3.7]{BC}. For (iv), $\rees{J}$ is Cohen-Macaulay by \cite[2.2]{Huneke}. For (v), $\rees{J}$ is Cohen-Macaulay if $\chr K = 0$ or $\chr K > \min\!\left\{2t,n-2t\right\}$ by \cite[3.4]{Baetica}.
\end{proof}

\bigskip

Using the characterization of $G_{\infty}$ from \cref{GsConditions}, one can show $I$ satisfies $G_{\infty}$ if and only if the conditions in (ii)-(v) are met or if $n \le m+1$ and the conditions of (i) are met. However, the conditions in (i) are less stringent than $G_{\infty}$ provided $n > m+1$. 

Based on a number of results for the generic case and combined with the above theorem, we obtain the results in the corollary below. Parts (i) and (iv) of \cref{LinFibCor} were already known from other work. In particular, from the work of Ap\'{e}ry \cite{Apery} or Gaeta  \cite{Gaeta} (for part (i)) and J. Watanabe \cite{Watanabe} (for part (iv)), it is known that $I$ is in the linkage class of a complete intersection. Thus, $I$ is strongly Cohen-Macaulay \cite[1.4]{Huneke3}. Hence, $I$ satisfies sliding depth. Any ideal which satisfies $G_{\infty}$ and sliding depth is of linear type \cite{HSV:LinearType}. To the best of our knowledge, all other parts of the below corollary are original.

\bigskip

\begin{cor}
\label{LinFibCor}
\hypertarget{LinFibCor}{}
Let $K$ be a field. Suppose $R = K\!\left[x_{1},\ldots,x_{d}\right]$ is a polynomial ring in $d$ indeterminates over $K$. 
\begin{enumerate}
   \item[$(\mathrm{i})$] Suppose $A$ is an $m \times \left(m+1\right)$ matrix with $1 \le m$, $I = \detid{m}{A}$ is of generic height, and $I$ satisfies $G_{\infty}$ $($i.e., $\hgt{\detid{j}{A}} \ge m-j+2$ for all $1 \le j \le m-1)$. Then $I$ is of linear type.
    \item[$(\mathrm{ii})$] Suppose $A$ is an $n \times n$ matrix with $2 \le n$, $I = \detid{n-1}{A}$ is of generic height, and $I$ satisfies $G_{\infty}$ $($i.e., $\hgt{\detid{j}{A}} \ge \left(n-j+1\right)^{2}$ for all $1 \le j \le n-2)$. Then $I$ is of linear type.
    \item[$(\mathrm{iii})$] Suppose $A$ is an $n \times n$ symmetric matrix with $2 \le n$, $I = \detid{n-1}{A}$ is of generic height, and $I$ satisfies $G_{\infty}$ $($i.e., $\hgt{\detid{j}{A}} \ge \binom{n-j+2}{2}$ for all $1 \le j \le n-2)$. Then $I$ is of linear type.
    \item[$(\mathrm{iv})$] Suppose $A$ is an $n \times n$ alternating matrix with $n$ odd and $3 \le n$, $I = \pfaffid{n-1}{A}$ is of generic height, and $I$ satisfies $G_{\infty}$ $($i.e., $\hgt{\pfaffid{2j}{A}} \ge n-2j+2$ for all $2 \le 2j \le n-3)$. Then $I$ is of linear type.
    \item[$(\mathrm{v})$] Suppose $\chr K \ne 2$, $A$ is an $n \times n$ alternating matrix with $n$ even and $4 \le n$, $I = \pfaffid{n-2}{A}$ is of generic height, and $I$ satisfies $G_{\infty}$ $($i.e.,  $\hgt{\pfaffid{2j}{A}} \ge \binom{n-2j+2}{2}$ for all $2 \le 2j \le n-4)$. Then $I$ is of linear type.
    \item[$(\mathrm{vi})$] Suppose $A$ is an $m \times n$ matrix with $1 \le m \le n$, the entries of $A$ are homogeneous of the same degree, $I = \detid{m}{A}$ is of generic height, and $\hgt{\detid{j}{A}} \ge \left(m-j+1\right)\left(n-m\right)+1$ for all $1 \le j \le m-1$. Then $I$ is of fiber type.
    \item[$(\mathrm{vii})$] Suppose $\chr K = 0$, $A$ is a $3 \times n$ matrix with $3 \le n$, the entries of $A$ are homogeneous of the same degree, $I = \detid{2}{A}$ is of generic height, and  $\hgt{\detid{1}{A}} \ge 3n$. Then $I$ is of fiber type.
\end{enumerate}
\end{cor}

\begin{proof}
By \cref{SpecializationRees}, we have $\rees{J} \otimes_{B} R \cong \rees{I}$ via the natural map for all cases. In particular, by \hyperref[SpecializationLemmaSeq]{\cref{SpecializationLemmaSeq}.iii}, the sequence
\[ \A{k}{J} \otimes_{B} R \to \A{k}{I}  \to 0 \]
is exact and homogeneous for each $k$.

If $J$ is of linear type, then $\A{k}{J} = 0$ for all $k$. Thus, $\A{k}{I} = 0$ for all $k$. Hence $I$ is of linear type.

By the following references, we know $J$ is of linear type. For (i), \cite[Proposition 1.1]{HunekeDSequence}. For (ii), \cite[2.6]{Huneke}. For (iii), \cite[2.10]{Kotzev}. For (iv), \cite[2.2]{Huneke3}. For (v),  when $2$ is invertible,  \cite[2.1]{Baetica2}.

To prove the fiber type results, we note that being generated in degree $0$ is preserved under homogeneous surjections. Therefore, $I$ is of fiber type if $J$ is of fiber type.

For (vi), it is known that $J$ is of fiber type by \cite[2.5]{EH}. For (vii), when $\chr K = 0$, it is known that $J$ is of fiber type from \cite[7.3]{HPPRS}.
\end{proof}

\bigskip

\section{Degree Bounds}
\label{SectionDegreeBounds}
\hypertarget{SectionDegreeBounds}{}

\numberwithin{thm}{subsection}

\subsection{Tools}
\label{DegBoundTools}
\hypertarget{DegBoundTools}{}

In this section, we are concerned with bounding the generation and concentration degrees of $\A{k}{I}$ for a determinantal or Pfaffian ideal $I$ of generic (symmetric, alternating) height. Our main tool is the homogeneous exact sequence of \hyperref[SpecializationLemmaSeq]{\cref{SpecializationLemmaSeq}.iii} and the results of \cite[5.4, 5.6]{KPU} which relate degree shifts in approximate resolutions to degree bounds. In particular, we use \cite[5.4, 5.6]{KPU} to obtain degree bounds of $\ker \psi\!_{k}$ using the approximate resolution in \hyperref[SpecializationLemmaApproxRes]{\cref{SpecializationLemmaApproxRes}.ii}.

We introduce the following data which will be referenced throughout this section.

\bigskip

\begin{data}
\label{GradedData}
\hypertarget{GradedData}{Let} $R$ be a standard graded polynomial ring in $d$ indeterminates over the field $K$, and suppose $1 \le t \le m \le n$.
\begin{enumerate}
   \item[$(\mathrm{a})$]\label{GradedDataOrd}
\hypertarget{GradedDataOrd}{} Let $A$ be an $m \times n$ matrix over $R$ with all entries homogeneous of the same degree $\delta$ and $I = \detid{t}{A}$ be of generic height. Let $X$ be a generic $m \hspace{0.5 mm}\times \hspace{0.5 mm} n$ matrix over $R$, $B=R\!\left[X\right]$ where $\deg X_{ij} = 1$ for all $X_{ij}$, and $J=\detid{t}{X}$. Further, for each $k$, let $\left(\comp[\bullet]{D^{k}},\comp{\varphi^{k}}\right)$ be the minimal homogeneous free $B$-resolution of $J^{k}\!\left(tk\right)$.
    \item[$(\mathrm{b})$]\label{GradedDataSym}
\hypertarget{GradedDataSym}{} Let $A$ be an $n \times n$ symmetric matrix over $R$ with all entries homogeneous of the same degree \\$\delta$ and $I = \detid{t}{A}$ be of generic symmetric height. Let $X$ be a generic symmetric $n \times n$ matrix over $R$, $B=R\!\left[X\right]$  where $\deg X_{ij} = 1$ for all $X_{ij}$, and $J=\detid{t}{X}$. Further, for each $k$, let $\left(\comp[\bullet]{D^{k}},\comp{\varphi^{k}}\right)$ be the minimal homogeneous free $B$-resolution of $J^{k}\!\left(tk\right)$.
    \item[$(\mathrm{c})$]\label{GradedDataAlt}
\hypertarget{GradedDataAlt}{} Let $A$ be an $n \times n$ alternating matrix over $R$ with all entries homogeneous of the same degree $\delta$ and $I = \pfaffid{2t}{A}$ be of generic alternating height. Let $X$ be a generic alternating $n \times n$ matrix over $R$, $B=R\!\left[X\right]$  where $\deg X_{ij} = 1$ for all $X_{ij}$, and $J=\pfaffid{2t}{X}$. Further, for each $k$, let $\,\left(\comp[\bullet]{D^{k}},\comp{\varphi^{k}}\right)$ be the minimal homogeneous free $B$-resolution of $J^{k}\!\left(tk\right)$.
\end{enumerate}
\end{data}

\bigskip

The following result is the main argument used to obtain degree bounds on $\mathcal{A}\!\left(I\right)$ from degree bounds on $\mathcal{A}\!\left(J\right)$ and from free resolutions of $J^{k}$.  

\bigskip

\begin{prop}
\label{Bounds}
\hypertarget{Bounds}
Suppose $R$ is a standard graded polynomial ring in $d > 0$ indeterminates over the field $K$, and let $A_{1},\ldots,A_{D}$ be a sequence in $R$ where each $A_{i}$ is homogeneous of the same degree $\,\delta$. Suppose $B = R\!\left[X_{1},\ldots,X_{D}\right]$ where $\deg X_{i} = \delta$ for each $i$. Let $J$ be a homogeneous $B$\!-ideal generated by forms of the same degree $q$. Define $Y_{i} = X_{i}-A_{i}$, assume $Y_{1}, \ldots, Y_{D}$ is regular on $B$ and $B/J$, and let $N = \left(Y_{1},\ldots,Y_{D}\right)$. We give $R$ the $B$-algebra structure induced by the homogeneous isomorphism $R \cong B/N$. Let $I = JR$.  For each $k$, let $\,\left(\comp{E^{k}},\comp{\tau^{k}}\right)$ be a homogeneous finite free $B$-resolution of $\,J^{k}\!\left(kq\right)$ where each $\comp[i]{E^{k}}$ is finitely generated. 
Suppose $\,\left\{K_{i}\right\}$ is a family of $B$-ideals so that  $K_{i} \subseteq \sqrt{I\left(\comp[i]{\varphi^{k}}\right)}$ for all $i$ and for all $k$. 
If $\;\hgt{K_{i}R} \ge \min\left\{i+1,d\right\}\,$ for all $i$, then 
    \[ \bzero{\A{k}{I}} \le \max\left\{ \bzero{\A{k}{J}}, \bzero{\comp[d-1]{E^{k}}}-d+1 \right\}, \mbox{ and }\]
    \[ \tdp{\A{k}{I}} \le \max\left\{ \tdp{\A{k}{J}}, \bzero{\comp[d]{E^{k}}}-d \right\}.\]
\end{prop}

\begin{proof}
Let $\psi\!_{k} \colon J^{k} \otimes_{B} R \to I^{k}$ be the natural surjection. 

By \hyperref[SpecializationLemmaSeq]{\cref{SpecializationLemmaSeq}.iii}, the homogeneous sequence of $R$-modules 
\[ \A{k}{J} \otimes_{B} R \to \A{k}{I} \to \left(\ker \psi\!_{k}\right)\!\left(tk\delta\right) \to 0 \]
is exact. Therefore, we have
\[ \bzero{\A{k}{I}} \le \max\left\{ \bzero{\A{k}{J}}, \bzero{\left(\ker \psi\!_{k}\right)\left(tk\delta\right)} \right\}, \mbox{ and }\]
\[ \tdp{\A{k}{I}} \le \max\left\{ \tdp{\A{k}{J}}, \tdp{\left(\ker \psi\!_{k}\right)\left(tk\delta\right)} \right\}.\]

It remains to show  $\bzero{\left(\ker \psi\!_{k}\right)\left(tk\delta\right)} \le \bzero{\comp[d-1]{E^{k}}}-d+1$ and $\tdp{\left(\ker \psi\!_{k}\right)\left(tk\delta\right)} \le \bzero{\comp[d]{E^{k}}}-d$.

We begin by showing $\ker \psi\!_{k} = \locoh{0}{\mf{m}}{J^{k} \otimes_{B} R}$ for each $k$. As in the proof of \cref{SpecializationRees}, $\ker \psi\!_{k}$ is the $R$-torsion of $J^{k} \otimes_{B} R$. By \cref{TorsionKernelLemma}, we have \[\ker \psi\!_{k} = \locoh{0}{\fm}{J^{k} \otimes_{B} R} \,\text{ if and only if }\; \pdim{R_{\fp}}{\left(J^{k} \otimes_{B} R \right)_{\fp}} < \hgt{\fp}\] for all primes $\fp \in \proj{R}$ with $\hgt{\fp} > 0$.
    
By our assumption that $\hgt{K_{i}R} \ge \min\left\{i+1,d\right\}$ for all $i$ and since $\hgt{K_{i}R} \le \hgt{I\left(\comp[i]{\tau^{k}} \otimes_{B} R\right)}$ for all $i$, we get $\hgt{I\left(\comp[i]{\tau^{k}} \otimes_{B} R\right)}\ge \min\left\{i+1,d\right\}$ for all $i$. Now, let $\fp \in \proj{R}$ with $\hgt{\fp} > 0$. Then \[\hgt{\left(I\left(\comp[\hgt{\fp}]{\tau^{k}} \otimes_{B} R\right)_{\fp}\right)} \ge \min\left\{\hgt{\fp}+1,d\right\} >  \dim R_{\fp}.\] Hence, $\pdim{R_{\fp}}{\left(J^{k} \otimes_{B} R \right)_{\fp}} < \hgt{\fp}$ for all $k$. Therefore, $\ker \psi\!_{k} = \locoh{0}{\mf{m}}{J^{k} \otimes_{B} R}$ for each $k$.
    
Since $\ker \psi\!_{k} = \locoh{0}{\mf{m}}{J^{k} \otimes_{B} R}$ for each $k$ and since $\comp{E^{k}} \! \otimes_{B} \! R$ is a homogeneous approximate resolution of $J^{k} \otimes_{B} R$ (see \hyperref[SpecializationLemmaApproxRes]{\cref{SpecializationLemmaApproxRes}.ii}), by \cite[5.4, 5.6]{KPU}, we have 
\[\bzero{\left(\ker \psi\!_{k}\right)\left(tk\delta\right)}  \le \bzero{\comp[d-1]{E^{k}} \otimes_{B} R}-d+1 \le \bzero{\comp[d-1]{E^{k}}}-d+1.\]
Likewise, we have
\[\tdp{\left(\ker \psi\!_{k}\right)\left(tk\delta\right)} \le \bzero{\comp[d]{E^{k}} \otimes_{B} R}-d \le \bzero{\comp[d]{E^{k}}}-d.\]
\end{proof}

\bigskip

We wish to apply \cref{Bounds} to the setting of \cref{GradedData}, where the indeterminates $X_{i}$ are the algebraically independent entries of the generic (symmetric, alternating) matrix $X$. However, the hypotheses in \cref{Bounds} state $\deg X_{ij} = \delta$, whereas \cref{GradedData} gives $\deg X_{ij} = 1$, which is the grading one typically uses to calculate quantities involving the ideals $J^{k}$. Since the generators of $J^{k}$ can be expressed entirely in the entries of the matrix $X$, transitioning from the grading in \cref{Bounds} to the grading in \cref{GradedData} corresponds to multiplying the degrees of homogeneous polynomials in the entries of $X$ by $\delta$. Therefore, we obtain the following corollary.

\bigskip

\begin{cor}
\label{BoundsRemark}
\hypertarget{BoundsRemark}
Adopt \hyperref[GradedDataOrd]{\cref{GradedDataOrd}.a}, \hyperref[GradedDataSym]{\cref{GradedDataSym}.b}, or \hyperref[GradedDataAlt]{\cref{GradedDataAlt}.c}, and suppose $\,\left\{K_{i}\right\}$ is a family of $B$-ideals so that  $K_{i} \subseteq \sqrt{I\left(\comp[i]{\varphi^{k}}\right)}$ for all $i$ and for all $k$.  If $\;\hgt{K_{i}R} \ge \min\left\{i+1,d\right\}\,$ for all $i$, then, for each $k$, one has
\[ \bzero{\A{k}{I}} \le \max\left\{ \delta \bzero{\A{k}{J}}, \delta \bzero{\comp[d-1]{D^{k}}}-d+1 \right\}, \mbox{ and }\]
    \[ \tdp{\A{k}{I}} \le \max\left\{ \delta \tdp{\A{k}{J}}, \delta \bzero{\comp[d]{D^{k}}}-d \right\}.\]
\end{cor}

\bigskip

The above result is fairly general, and the next remark shows how one can apply it to determinantal and Pfaffian ideals of generic height.

\bigskip

\begin{cor}[]
\label{DegBoundRemark}
\hypertarget{DegBoundRemark}{}
Suppose one of the following sets of hypotheses is satisfied.

\begin{enumerate}
    \item[$(\mathrm{i})$] Adopt \hyperref[GradedDataOrd]{\cref{GradedDataOrd}.a}, and let $\,t = m$. Suppose \[\hgt{\detid{j}{A}} \ge \min\left\{\left(m-j+1\right)\left(n-m\right)+1,d\right\}\] for all $\;1 \le j \le m-1$.
    
    \item[$(\mathrm{ii})$] Adopt \hyperref[GradedDataOrd]{\cref{GradedDataOrd}.a}, and let $\,1 \le t < m$. Suppose \[\hgt{\detid{j}{A}} \ge \min\left\{\left(m-j+1\right)\left(n-j+1\right),d\right\}\] for all $\;1 \le j \le m-1$.
    
    \item[$(\mathrm{iii})$] Adopt \hyperref[GradedDataSym]{\cref{GradedDataSym}.b}. Suppose \[\hgt{I_{j}\!\left(A\right)} \ge \min\left\{\binom{n-j+2}{2},d\right\}\] for all $\;1 \le j \le t-1$.
    
    \item[$(\mathrm{iv})$] Adopt \hyperref[GradedDataAlt]{\cref{GradedDataAlt}.c}, and let $\;2t = n-1$. Suppose  \[\hgt{\pfaffid{2j}{A}} \ge \min\left\{n-2j+2,d\right\}\] for all $\;1 \le j \le t-1$.
    
    \item[$(\mathrm{v})$] Adopt \hyperref[GradedDataAlt]{\cref{GradedDataAlt}.c}, and let $\;2 \le 2t < n-1$. Suppose  \[\hgt{\pfaffid{2j}{A}} \ge \min\left\{\binom{n-2j+2}{2},d\right\}\] for all $\;1 \le j \le t-1$.
\end{enumerate}
Then, for each $k$, one has
\[ \bzero{\A{k}{I}} \le \max\left\{ \delta \bzero{\A{k}{J}}, \delta \bzero{\comp[d-1]{D^{k}}}-d+1 \right\}, \mbox{ and }\]
    \[ \tdp{\A{k}{I}} \le \max\left\{ \delta \tdp{\A{k}{J}}, \delta \bzero{\comp[d]{D^{k}}}-d \right\}.\]
\end{cor}

\begin{proof}
By \cref{BoundsRemark}, it suffices to find a family $\left\{K_{i}\right\}$ of $B$-ideals so that  $K_{i} \subseteq \sqrt{I\left(\comp[i]{\varphi^{k}}\right)}$ for all $i$ and for all $k$ and so that $\;\hgt{K_{i}R} \ge \min\left\{i+1,d\right\}\,$ for all $i$.  To handle all cases simultaneously, we establish the following notation.
\begin{enumerate}
    \item[$(\mathrm{i})$] Let $I_{j} = \detid{j}{X}$ and $\sigma\!\left(j\right) = \left(m-j\right)\left(n-m\right)+1$.
    
    \item[$(\mathrm{ii})$] Let $I_{j} = \detid{j}{X}$ and $\sigma\!\left(j\right) = \left(m-j\right)\left(n-j\right)$.
    
    \item[$(\mathrm{iii})$] Let $I_{j} = \detid{j}{X}$ and $\sigma\!\left(j\right) = \binom{n-j+1}{2}$.
    
    \item[$(\mathrm{iv})$] Let $I_{j} = \pfaffid{2j}{X}$ and $\sigma\!\left(j\right) = n-2j$.
    
    \item[$(\mathrm{v})$] Let $I_{j} = \pfaffid{2j}{X}$ and $\sigma\!\left(j\right) = \binom{n-2j}{2}$.
\end{enumerate}

For each $1 \le i \le \hgt{I}-1$, we set $K_{i} = \sqrt{J}$.  For each $\hgt{I} \le i \le \max_{k}\!\left\{\pdim{}{J^{k}}\right\}$, we let $j_{0}$ be the smallest $j$ satisfying $i \ge \sigma\!\left(j\right)$, and set $K_{i} = \sqrt{I_{j_{0}}}$. 

We now show that for each $i$ satisfying $\hgt{I} \le i \le \max_{k}\left\{ \pdim{}{J^k}\right\}$ we have $1 \le j_{0} \le t-1$, which gives us control over $\hgt{K_{i}R}$ by assumption. Since $I$ is of generic (symmetric, alternating) height, $\sigma\!\left(t-1\right) = \hgt{I} \le i$. Thus, by the definition of $j_{0}$, we have $j_{0} \le t-1$. By \cref{CNAnalyticSpread}, we see $\sigma\!\left(0\right) =  \max_{k}\!\left\{\pdim{}{J^{k}}\right\}+1 > i$. Thus, by the definition of $j_{0}$, we have $1 \le j_{0}$.

Fix $i$ with $1 \le i \le \hgt{I}-1$.  Note that we can extend $\left(\comp{D^{k}},\comp{\varphi^{k}}\right)$ to a resolution of $B/J^{k}$ by taking $0 \to \comp{D^{k}} \overset{\comp[0]{\varphi^{k}}}{\to} B$.  Then, since $\ann{}{B/J^{k}} \ne 0$ and $B$ is Cohen-Macaulay, we have $\sqrt{J^{k}} = \sqrt{\detid{}{\comp[0]{\varphi^{k}}}} = \cdots = \sqrt{\detid{}{\comp[\hgt{J^{k}-1}]{\varphi^{k}}}}$. Since $I$ is of generic height, $\hgt{J^{k}} = \hgt{J} = \hgt{I}$.  In particular, we have $K_{i} = \sqrt{J} = \sqrt{J^{k}} \subseteq \sqrt{\detid{}{\comp[i]{\varphi^{k}}}}$, and $\hgt{K_{i}R} > i$. Hence, $\hgt{K_{i}R} \ge \min\!\left\{i+1,d \right\}$.

Now, fix $i$ with $\hgt{I} \le i \le \max_{k}\!\left\{\pdim{}{J^{k}}\right\}$, and recall $j_{0}$ satisfies $\sigma\!\left(j_{0}-1\right) > i \ge \sigma\!\left(j_{0}\right)$ and $1 \le j_{0} \le t-1$. Since $i \ge \sigma\!\left(j_{0}\right)$, by \cref{ContainmentLemma}, we have $K_{i} = \sqrt{I_{j_{0}}} \subseteq \sqrt{\detid{}{\comp[i]{\varphi^{k}}}}$.  Moreover, by assumption, $\hgt{K_{i}R} = \hgt{I_{j_{0}}}R = \hgt{\detid{j_{0}}{A}} \ge \min\!\left\{\sigma\!\left(j_{0}-1\right),d\right\} \ge \min\!\left\{i+1,d\right\}$.
\end{proof}

\bigskip

In much of the literature for defining equations of Rees rings, one requires the condition $G_{d}$ in order to obtain results about degree bounds. We note that in the above corollary, $G_{d}$ is satisfied in part (i) when $n = m+1$, in part (ii) when $n =m$ and $t = n-1$, in part (iii) when $t = n-1$, in part (iv), and in part (v) when $2t = n-2$. However, in each other case, the height bounds required are strictly weaker conditions than $G_{d}$. This is useful because it is not always possible to satisfy $G_{d}$ in the generic setting.

\subsection{Degree Bounds for Ordinary Matrices}
\label{OrdSection}
\hypertarget{OrdSection}{}

In this subsection, we develop bounds on $\bzero{\A{k}{I}}$ and $\tdp{\A{k}{I}}$ when $I$ is the ideal of minors of an ordinary matrix. To use the tools we have previously developed, we must utilize bounds on $\bzero{\A{k}{J}}$ and $\tdp{\A{k}{J}}$, where $J$ is the corresponding determinantal ideal of a generic ordinary matrix.  \Cref{GlobalConditionsOrdinary} provides known results in this regard.

\bigskip

\begin{prop}
\label[proposition]{GlobalConditionsOrdinary}
\hypertarget{GlobalConditionsOrdinary}{}
Let, $t$, $m$, and $n$ be integers satisfying $1 \le t \le m \le n$, $K$ be a field, $X$ be an $m \times n$ generic ordinary matrix over $K$, $S = \poly{K}{X}$, and $J = \detid{t}{X}$.
\begin{enumerate}
    \item If $t = 1$, then $J$ is of linear type. Hence, for all $k$, \[\bzero{\A{k}{J}} = \tdp{\A{k}{J}} = -\infty.\]
    \item {\normalfont \cite[Proposition 1.1]{HunekeDSequence}} If \,$n \le m+1$ and $t = m$, then $J$ is of linear type. Hence, for all $k$, \[\bzero{\A{k}{J}} = \tdp{\A{k}{J}} = -\infty.\]
    \item If $\,n \ge m+2$ and $t = m$, then $J$ is not of linear type on the punctured spectrum of $S$. Hence, for some $k$, \[\tdp{\A{k}{J}} = \infty.\]
    \item {\normalfont \cite[2.6]{EH}} If $\,t = m$, then $J$ is of fiber type. Hence, for all $k$, \[\bzero{\A{k}{J}} \le 0.\] 
    \item {\normalfont \cite[2.4]{Huneke}} If $\,n = m$ and $t = n-1$, then $J$ is of linear type. Hence, for all $k$,  \[\bzero{\A{k}{J}} = \tdp{\A{k}{J}} = -\infty.\]
    \item {\normalfont \cite[7.3]{HPPRS}} If $\,\chr K = 0$, $m = 3$, $n \ge 3$, and $t = 2$, then $J$ is of fiber type. Hence, \[\bzero{\A{k}{J}} \le 0.\] 
    \item If $\,t = 2$, then $J$ is of linear type on the punctured spectrum of $S$. Hence, for all $k$, \[\tdp{\A{k}{J}} < \infty.\]
    \item If $\,2 < t < m$ and it is not the case that $t+1=m=n$, then $J$ is not of linear type on the punctured spectrum of $S$. Hence, for some $k$, \[\tdp{\A{k}{J}} = \infty.\]
\end{enumerate}
\end{prop}

\bigskip

Some explicit minimal free resolutions of $J^{k}$ are known. For example, we have already stated minimal free resolutions in the case of the maximal minors of an ordinary matrix in \cref{ResolutionPowersHilbertBurch}. Note that this gives us $\bzero{\comp[i]{F^{k}}} = i$ for all $i$ with $1 \le i \le \min\!\brces{k,m}\paren{n-m}$ and $\bzero{\comp[i]{F^{k}}} = -\infty$ for all $i$ with $i > \min\!\brces{k,m}\paren{n-m}$.

\bigskip

\begin{thm}[Maximal Minors of an Ordinary Matrix]
\label[theorem]{TheoremMaximalMinors}
\hypertarget{TheoremMaximalMinors}{}
Let $K$ be a field, $R = \poly{K}{x_{1},\ldots,x_{d}}$ be a polynomial ring in $d > 0$ indeterminates over $K$, $m$ and $n$ be integers satisfying $1 \le m \le n$, $A$ be an $m \times n$ matrix whose entries are all homogeneous elements of $R$ of the same degree $\delta$, and $I = \detid{m}{A}$ be of generic height $($that is, $\hgt{I} = n-m+1)$. Further, suppose \[\hgt{\detid{j}{A}} \ge \min\!\brces{\paren{m-j+1}\paren{n-m}+1,d}\] for all $j$ in the range $1 \le j \le m-1$.
\begin{enumerate}
    \item Let $n = m$. Then , for all $k$,
    \[\bzero{\A{k}{I}} = \tdp{\A{k}{I}} = -\infty.\]
    
    \item Let $n = m+1$.

    If $d > \min\!\brces{k,m}$, then  
    \[\bzero{\A{k}{I}} = \tdp{\A{k}{I}} =  -\infty.\]

    If $d \le \min\!\brces{k,m}$, then  \[\bzero{\A{k}{I}} \le \paren{d-1}\paren{\delta -1} \text{ and } \tdp{\A{k}{I}} \le d\paren{\delta-1}.\]
    
    \item Let $n \ge m+2$.

    If $d-1 > \min\!\brces{k,m}\paren{n-m}$, then  
    \[\bzero{\A{k}{I}} \le 0.\] 
    
    If $d-1 \le \min\!\brces{k,m}\paren{n-m}$, then  
    \[\bzero{\A{k}{I}} \le  \paren{d-1}\paren{\delta -1}.\] 
\end{enumerate}
\end{thm}

\begin{proof}

For parts (b) and (c) we apply \cref{DegBoundRemark}.  In particular, for each $k$, we have 
\[ \bzero{\A{k}{I}} \le \max\!\brces{ \delta \bzero{\A{k}{J}}, \delta \bzero{\comp[d-1]{D^{k}}}-d+1 } , \text{ and}\]
\[ \tdp{\A{k}{I}} \le \max\!\brces{ \delta \tdp{\A{k}{J}}, \delta \bzero{\comp[d]{D^{k}}}-d }.\]

We now prove (b). 

By part (b) of \cref{GlobalConditionsOrdinary}, $\bzero{\A{k}{J}} = \tdp{\A{k}{J}} = -\infty$ for all $k$. Therefore,
\[ \bzero{\A{k}{I}} \le  \delta \bzero{\comp[d-1]{D^{k}}}-d+1 , \text{ and}\]
\[ \tdp{\A{k}{I}}  \le  \delta \bzero{\comp[d]{D^{k}}}-d .\]

By \cref{ResolutionPowersHilbertBurch}, if $d > \min\!\brces{k,m}\paren{n-m}$, then $\bzero{\comp[d]{D^{k}}} = -\infty$. Therefore $\tdp{\A{k}{I}} \le -\infty$. Hence $\bzero{\A{k}{I}} = -\infty$. Likewise, if $d \le \min\!\brces{k,m}\paren{n-m}$, then $\bzero{\comp[d]{D^{k}}} = d$. Therefore, we have $\tdp{\A{k}{I}} \le d\paren{\delta-1}$ and $\bzero{\comp[d-1]{D^{k}}} = d-1$. Therefore, $\bzero{\A{k}{I}} \le \paren{d-1}\paren{\delta-1}$. 

We now prove (c). By part (c) of \cref{GlobalConditionsOrdinary}, $\tdp{\A{k}{J}} = \infty$ for some $k$, so we do not draw any conclusions about $\tdp{\A{k}{I}}$ for any $k$. On the other hand, by part (d) of \cref{GlobalConditionsOrdinary}, we know $\bzero{\A{k}{J}} \le 0$ for all $k$. Therefore, we have \[ \bzero{\A{k}{I}} \le \max\!\brces{0, \delta \bzero{\comp[d-1]{D^{k}}}-d+1}.\]

By \cref{ResolutionPowersHilbertBurch}, if $d-1 > \min\!\brces{k,m}\paren{n-m}$, then $\bzero{\comp[d-1]{D^{k}}} = -\infty$. Thus, $\bzero{\A{k}{I}} \le 0$.  On the other hand, if $d-1 \le \min\!\brces{k,m}\paren{n-m}$, then $\bzero{\comp[d-1]{D^{k}}} = d-1$. Therefore, $\bzero{\A{k}{I}} \le \paren{d-1}\paren{\delta-1}$.
\end{proof}

\bigskip

\begin{cor}
Let $K$ be a field, $R = \poly{K}{x_{1},\ldots,x_{d}}$ be a polynomial ring in $d > 0$ indeterminates over $K$. Let $m$ and $n$ be integers satisfying $1 \le m \le n$, and let $A$ be an $m \times n$ matrix whose entries are all homogeneous elements of $R$ of the same degree $\delta$. Suppose $I = \detid{m}{A}$ is of generic height $($that is, $\hgt{I} = n-m+1)$, and suppose $\,\hgt{\detid{j}{A}} \ge \min\!\brces{\paren{m-j+1}\paren{n-m}+1,d}$ for all $j$ in the range $1 \le j \le m-1$.
\begin{enumerate}
    \item If $\,\delta = 1$, then $I$ is of fiber type. 
    \item If $\,n=m+1$ and $d > m$, then $I$ is of linear type.
    \item If $\,n=m+1$, $d \le m$, and $\,\delta = 1$, then $\paren{x_{1},\ldots,x_{d}}\A{}{I} = 0$.
    \item If $\,n \ge m+2\,$ and $\,d > m\paren{n-m}+1$, then $I$ is of fiber type.
\end{enumerate}
\end{cor}

\bigskip

Parts (a), (b), and (c) are known results which we recover. Part (a) was proven by Bruns, Conca, and Varbaro in \cite[3.7]{BCV} using techniques from representation theory. Part (b) can be proven in the following way: first, note the requirements in the corollary and in part (b) imply $I$ satisfies $G_{\infty}$. Then, from the work of Ap\'{e}ry \cite{Apery} or of Gaeta \cite{Gaeta}, it is known that $I$ is in the linkage class of a complete intersection; thus, $I$ is strongly Cohen-Macaulay \cite[1.4]{Huneke3}. Hence, $I$ satisfies sliding depth. Finally, any ideal which satisfies $G_{\infty}$ and sliding depth is of linear type \cite{HSV:LinearType}. Part (c) follows from the work of Kustin, Polini, and Ulrich  \cite[6.1.a]{KPU}, where they applied the same technique we used here.  We are unaware of whether part (d) is known from other methods.

For other ideals of minors of ordinary matrices, explicit resolutions of $J^{k}$ are unknown. However, in certain cases, we have information regarding the regularities of these ideals. Thus, we may still obtain degree bounds.

\bigskip

\begin{thm}[$2 \times 2$ Minors]
\label[theorem]{Theorem2x2Minors}
\hypertarget{Theorem2x2Minors}{}
Let $K$ be a field of characteristic zero. Let $R = \poly{K}{x_{1},\ldots,x_{d}}$ be a polynomial ring in $d > 0$ indeterminates over $K$. Let $m$ and $n$ be integers satisfying $2 \le m \le n$ and let $A$ be an $m \times n$ matrix whose entries are all homogeneous elements of $R$ of the same degree $\delta$. Suppose $I = \detid{2}{A}$ is of generic height $($that is, $\hgt{I} = \paren{m-1}\paren{n-1})$, and let $X$ be an $m \times n$ generic matrix over $R$ and $J = \detid{2}{X}$. Further, suppose \hspace{0.08mm} $\hgt{\detid{1}{A}} \ge \min\!\brces{mn,d}$.
\begin{enumerate}
    \item Let $m = 3$. Then, for all $k$,
    \[\bzero{\A{k}{I}} \le \paren{d-1}\paren{\delta-1}, \text{ and }\] \[\tdp{\A{k}{I}} \le  \max\!\brces{\delta \tdp{\A{k}{J}}, d\paren{\delta-1}}.\]
    
    \item

    If \, $2 \le k \le m-2$, then  
    \[\bzero{\A{k}{I}} \le \max\!\brces{\delta \bzero{\A{k}{J}},  \paren{d-1}\paren{\delta-1} + \delta \paren{m-k-1}}, \text{ and }\] 
    \[\tdp{\A{k}{I}} \le \max\!\brces{\delta \tdp{\A{k}{J}},  d\paren{\delta-1} + \delta \paren{m-k-1}}.\]
    
    If \, $k \ge m-1$, then
    \[\bzero{\A{k}{I}} \le \max\!\brces{\delta \bzero{\A{k}{J}},  \paren{d-1}\paren{\delta-1} }, \text{ and }\] 
    \[\tdp{\A{k}{I}} \le \max\!\brces{\delta \tdp{\A{k}{J}},  d\paren{\delta-1}}.\]
\end{enumerate}
\end{thm}

\begin{proof}
We apply \cref{DegBoundRemark}.  In particular, for each $k$, we have 
\[ \bzero{\A{k}{I}} \le \max\!\brces{ \delta \bzero{\A{k}{J}}, \delta \bzero{\comp[d-1]{D^{k}}}-d+1 } , \text{ and}\]
\[ \tdp{\A{k}{I}} \le \max\!\brces{ \delta \tdp{\A{k}{J}}, \delta \bzero{\comp[d]{D^{k}}}-d }.\]

Furthermore, by part (g) of \cref{GlobalConditionsOrdinary}, $\tdp{\A{k}{J}} < \infty$ for all $k$. Thus, we can place meaningful bounds on $\tdp{\A{k}{I}}$. 

Since $\compp{D^{k}}$ is the minimal homogeneous free resolution of $J^{k}\!\paren{2k}$, $\bzero{\compp[d-1]{D^{k}}} \le \REG J^{k}\!\paren{2k} +d-1$ and $\bzero{\compp[d]{D^{k}}} \le \REG J^{k}\!\paren{2k} +d$. Therefore,
\[ \delta \bzero{\compp[d-1]{D^{k}}}-d+1 \le \delta \REG J^{k}\!\paren{2k} + \paren{d-1}\paren{\delta-1}, \text{ and}\]
\[ \delta \bzero{\compp[d]{D^{k}}}-d \le \delta \REG J^{k}\!\paren{2k} + d\paren{\delta-1}.\]

Note $\REG J^{k}\!\paren{2k} = \REG J^{k}-2k$. 

By \cite[Theorem on Regularity]{Raicu}, if $2 \le k \le m-2$, then $\REG J^{k} = k+m-1$. Hence, $\REG J^{k}\!\paren{2k} = m-k-1$. Thus, we conclude that for $2 \le k \le m-2$,
\[\bzero{\A{k}{I}} \le \max\!\brces{\delta \bzero{\A{k}{J}},  \paren{d-1}\paren{\delta-1} + \delta \paren{m-k-1}}, \text{ and }\] 
\[\tdp{\A{k}{I}} \le \max\!\brces{\delta \tdp{\A{k}{J}},  d\paren{\delta-1} + \delta \paren{m-k-1}}.\]
    
By \cite[Theorem on Regularity]{Raicu}, if $k \ge m-1$, then $\REG J^{k} = 2k$. Therefore, given $k \ge m-1$,
\[\bzero{\A{k}{I}} \le \max\!\brces{\delta \bzero{\A{k}{J}},  \paren{d-1}\paren{\delta-1}}, \text{ and }\] 
\[\tdp{\A{k}{I}} \le \max\!\brces{\delta \tdp{\A{k}{J}},  d\paren{\delta-1}}.\]

This concludes the proof of (b). The proof of (a) is a special case.

In the setting of (a), $m = 3$. Thus, $k \ge m-1$ if and only if $k \ge 2$. Therefore, we restrict our attention to the case where $k \ge m-1$ above. By part (f) of \cref{GlobalConditionsOrdinary}, we have $\bzero{\A{k}{J}} \le 0$ for all $k$. Hence, for all $k$, 
\[\bzero{\A{k}{I}} \le  \paren{d-1}\paren{\delta-1}, \text{ and }\] 
\[\tdp{\A{k}{I}} \le \max\!\brces{\delta \tdp{\A{k}{J}},  d\paren{\delta-1}}.\]
\end{proof}

\bigskip

\begin{cor}
\label[cor]{FiberTypeSize2}
\hypertarget{FiberTypeSize2}{}
Let $K$ be a field of characteristic zero. Let $R = \poly{K}{x_{1},\ldots,x_{d}}$ be a polynomial ring in $d > 0$ indeterminates over $K$. Let $n$ be an integer satisfying $3 \le n$, and let $A$ be a $3 \times n$ matrix whose entries are all homogeneous elements of $R$ of degree $1$. Suppose $I = \detid{2}{A}$ is of generic height $($that is, $\hgt{I} = 2\paren{n-1})$ and suppose \hspace{0.08mm} $\hgt{\detid{1}{A}} \ge \min\!\brces{3n,d}$. Then $I$ is of fiber type.
\end{cor}

\bigskip

In other cases, Raicu is not able to compute the regularity for all powers of the ideal $J$, but is able to compute the regularity for sufficiently high powers.

\bigskip

\begin{thm}[Submaximal Minors of a Square Matrix]
\label[theorem]{TheoremSubmaximalMinorsOrdinary}
\hypertarget{TheoremSubmaximalMinorsOrdinary}{}
Let $K$ be a field of characteristic zero. Let $R = \poly{K}{x_{1},\ldots,x_{d}}$ be a polynomial ring in $d > 0$ indeterminates over $K$. Let $n$ be an integer satisfying $\,2 \le n$, and let $A$ be an $n \times n$ matrix whose entries are all homogeneous elements of $R$ of the same degree $\delta$. Suppose $I = \detid{n-1}{A}$ is of generic height $($that is, $\hgt{I} = 4)$ and suppose $\hgt{\detid{j}{A}} \ge \min\!\brces{\paren{n-j+1}^{2},d}$ for all $j$ in the range $1 \le j \le n-2$.

If $\,k \ge n-1$, then
\[ \bzero{\A{k}{I}} \le  \paren{d-1}\paren{\delta-1}+\delta N\!\paren{n}  \text{, and } \]\[ \tdp{\A{k}{I}} \le  d\paren{\delta-1}+\delta N\!\paren{n}, \]
    where $N\!\paren{n} = \begin{cases}
\paren{\frac{n-2}{2}}^{2} & \text{ if } n \text{ is even}\\
\frac{\paren{n-3}\paren{n-1}}{4} & \text{ if } n \text{ is odd.}
\end{cases}$
\end{thm}

\begin{proof}
We apply \cref{DegBoundRemark}. By part (e) of \cref{GlobalConditionsOrdinary}, $\bzero{\A{k}{J}} = \tdp{\A{k}{J}} = -\infty$ for all $k$. Therefore, we have
\[ \bzero{\A{k}{I}} \le  \delta \bzero{\comp[d-1]{F^{k}}}-d+1 , \text{ and}\]
\[ \tdp{\A{k}{I}}  \le  \delta \bzero{\comp[d]{F^{k}}}-d .\]

Since $\compp{D^{k}}$ is the minimal homogeneous free resolution of $J^{k}\!\paren{k\paren{n-1}}$, we have $\bzero{\compp[d-1]{F^{k}}} \le \REG J^{k}\!\paren{k\paren{n-1}} +d-1$ and $\bzero{\compp[d]{F^{k}}} \le \REG J^{k}\!\paren{k\paren{n-1}} +d$. Therefore,
\[ \delta \bzero{\compp[d-1]{F^{k}}}-d+1 \le \delta \REG J^{k}\!\paren{k\paren{n-1}} + \paren{d-1}\paren{\delta-1}, \text{ and}\]
\[ \delta \bzero{\compp[d]{F^{k}}}-d \le \delta \REG J^{k}\!\paren{k\paren{n-1}} + d\paren{\delta-1}.\]

By \cite[Theorem on Regularity]{Raicu}, if $k \ge n-1$, then \[\REG J^{k} = k\paren{n-1}+\begin{cases}
\paren{\frac{n-2}{2}}^{2} & \text{ if } n \text{ is even}\\
\frac{\paren{n-3}\paren{n-1}}{4} & \text{ if } n \text{ is odd.}
\end{cases}\] 

Let $N\!\paren{n} = \begin{cases}
\paren{\frac{n-2}{2}}^{2} & \text{ if } n \text{ is even}\\
\frac{\paren{n-3}\paren{n-1}}{4} & \text{ if } n \text{ is odd.}
\end{cases}$

Therefore, we conclude that for $k \ge n-1$,
\[\bzero{\A{k}{I}} \le  \paren{d-1}\paren{\delta-1} + \delta N\!\paren{n}, \text{ and }\] 
\[\tdp{\A{k}{I}} \le  d\paren{\delta-1} + \delta N\!\paren{n}.\]
\end{proof}

\bigskip

\begin{cor}
Let $K$ be a field of characteristic zero. Let $R = \poly{K}{x_{1},\ldots,x_{d}}$ be a polynomial ring in $d > 0$ indeterminates over $K$, and let $A$ be a $3 \times 3$ matrix whose entries are all homogeneous elements of $R$ of degree $1$. Suppose $I = \detid{2}{A}$ is of generic height $($that is, $\hgt{I} = 4)$, and suppose  $\hgt{\detid{1}{A}} \ge \min\!\brces{9,d}$. Then $I$ is of fiber type and $\paren{x_{1},\ldots,x_{d}}\A{}{I} = 0$.
\end{cor}

\bigskip

\begin{thm}[Minors of an Ordinary Matrix]
\label[theorem]{TheoremArbitraryMinors}
\hypertarget{TheoremArbitraryMinors}{}
Let $K$ be a field of characteristic zero. Let $R = \poly{K}{x_{1},\ldots,x_{d}}$ be a polynomial ring in $d > 0$ indeterminates over $K$. Suppose $t$, $m$, and $n$ are integers satisfying $2 < t < m \le n$, and let $A$ be an $m \times n$ matrix whose entries are all homogeneous elements of $R$ of the same degree $\delta$. Suppose $I = \detid{t}{A}$ is of generic height $($that is, $\hgt{I} = \paren{m-t+1}\paren{n-t+1})$. Let $X$ be an $m \times n$ generic matrix over $R$ and $J = \detid{t}{X}$. Further, suppose \hspace{0.08mm} $\hgt{\detid{j}{A}} \ge \min\!\brces{\paren{m-j+1}\paren{n-j+1},d}$ for all $j$ in the range $1 \le j \le t-1$.

If \, $k \ge m-1$, then
\[ \bzero{\A{k}{I}} \le  \max\brces{\delta\bzero{\A{k}{J}}, \paren{d-1}\paren{\delta-1}+\delta N\!\paren{t}},  \]
    where $N\!\left(t\right) = \begin{cases}
    \left(\frac{t-1}{2}\right)^{2}, & t \text{ is odd}\\
    \frac{\left(t-2\right)t}{4}, & t \text{ is even.}\\
    \end{cases}$
\end{thm}

\begin{proof}
By part (h) of \cref{GlobalConditionsOrdinary}, $\tdp{\A{k}{J}} = \infty$ for some $k$, so we do not draw conclusions about $\tdp{\A{k}{I}}$ for any $k$.

We apply \cref{DegBoundRemark} together with \cite[Theorem on Regularity]{Raicu} whence $k \ge m-1$ implies \[\REG J^{k} = tk+\begin{cases}
\paren{\frac{t-1}{2}}^{2} & \text{ if } t \text{ is odd}\\
\frac{\paren{t-2}t}{4} & \text{ if } t \text{ is even.}
\end{cases}\] 

Hence we define $N\!\paren{t} = \begin{cases}
\paren{\frac{t-1}{2}}^{2} & \text{ if } t \text{ is odd}\\
\frac{\paren{t-2}t}{4} & \text{ if } t \text{ is even.}
\end{cases}$

Then for $k \ge n-1$,
\[ \delta \bzero{\compp[d-1]{F^{k}}}-d+1 \le \delta N\!\paren{t} + \paren{d-1}\paren{\delta-1}.\]

Thus, we conclude that for $k \ge n-1$,
\[ \bzero{\A{k}{I}} \le  \max\brces{\delta\bzero{\A{k}{J}}, \paren{d-1}\paren{\delta-1}+\delta N\!\paren{t}}.  \]
\end{proof}

\bigskip

\subsection{Degree Bounds for Symmetric Matrices}
\label{SymSection}
\hypertarget{SymSection}{}

In this subsection, we consider the case where $I$ is the ideal of minors of a symmetric matrix. As in \cref{OrdSection}, we begin with known results on $\bzero{\A{k}{J}}$ and $\tdp{\A{k}{J}}$, where $J$ is the corresponding determinantal ideal of a generic symmetric matrix.

\bigskip

\begin{prop}
\label[proposition]{GlobalConditionsSymmetric}
\hypertarget{GlobalConditionsSymmetric}{}
Let $t$ and $n$ be integers satisfying $1 \le t \le n$, $K$ be a field, $X$ be an $n \times n$ generic symmetric matrix over $K$, $S = \poly{K}{X}$, and $J = \detid{t}{X}$.
\begin{enumerate}
    \item If \,$t = 1$, then $J$ is of linear type. Hence, for all $k$, \[\bzero{\A{k}{J}} = \tdp{\A{k}{J}} = -\infty.\]
    \item If \,$t = n$, then $J$ is of linear type. Hence, for all $k$, \[\bzero{\A{k}{J}} = \tdp{\A{k}{J}} = -\infty.\]
    \item {\normalfont \cite[2.10]{Kotzev}} If $t = n-1$, then $J$ is of linear type. Hence, for all $k$, \[\bzero{\A{k}{J}} = \tdp{\A{k}{J}} = -\infty.\]
    \item If \,$t = 2$, then $J$ is of linear type on the punctured spectrum. Hence, for all $k$, \[\tdp{\A{k}{J}} < \infty.\]
    \item If \,$2 < t < n-1$, then $J$ is not of linear type on the punctured spectrum. Hence, for some $k$, \[\tdp{\A{k}{J}} = \infty.\]
\end{enumerate}
\end{prop}

\bigskip

Unfortunately, we are not aware of any work which has been done to compute resolutions of $J^{k}$, nor are we aware of any work which has been done to calculate the regularity of $J^{k}$. Therefore, we are not able to obtain explicit degree bounds using our method. However, using the fact that $J$ is of linear type when $t = n-1$, we can state the following result for submaximal minors.

\bigskip

\begin{prop}[Submaximal Minors]
\label[prop]{SymmetricSubmaximal}
\hypertarget{SymmetricSubmaximal}{}
Let $K$ be a field. Let $R = \poly{K}{x_{1},\ldots,x_{d}}$ be a polynomial ring in $d > 0$ indeterminates over $K$. Suppose $n$ is an  integer satisfying $2 \le n$, and let $A$ be an $n \times n$ symmetric matrix whose entries are all homogeneous elements of $R$ of the same degree $\delta$. Suppose $I = \detid{n-1}{A}$ be of generic symmetric height $($that is, $\hgt{I} = 3)$, and suppose $\hgt{\detid{j}{A}} \ge \min\!\brces{\binom{n-j+2}{2},d}$ for all $j$ in the range $1 \le j \le n-2$. Then

\[ \bzero{\A{k}{I}} \le  \delta \bzero{\comp[d-1]{F^{k}}}-d+1  , \text{ and}\]
\[ \tdp{\A{k}{I}} \le  \delta \bzero{\comp[d]{F^{k}}}-d .\]

\end{prop}

\begin{proof}
Apply \cref{DegBoundRemark} with part (c) of \cref{GlobalConditionsSymmetric}, whence $\bzero{\A{k}{J}} = -\infty$ and $\tdp{\A{k}{J}} = -\infty$ for all $k$.
\end{proof}

\bigskip

\subsection{Degree Bounds for Alternating Matrices}
\label{AltSection}
\hypertarget{AltSection}{}

In this subsection, we consider the case where $I$ is the ideal of Pfaffians of an alternating matrix. As in \cref{OrdSection}, we begin with known results on $\bzero{\A{k}{J}}$ and $\tdp{\A{k}{J}}$, where $J$ is the corresponding Pfaffian ideal of a generic alternating matrix.

\bigskip

\begin{prop}
\label[proposition]{GlobalConditionsAlternating}
\hypertarget{GlobalConditionsAlternating}{}
Let $t$ and $n$ be integers satisfying $2 \le 2t \le n$, $K$ be a field, $X$ be an $n \times n$ generic alternating matrix over $K$, $S=\poly{K}{X}$, and $J = \pfaffid{2t}{X}$.
\begin{enumerate}
    \item If $\;2t = 2$, then $J$ is of linear type. Hence, for all $k$, \[\bzero{\A{k}{J}} = \tdp{\A{k}{J}} = -\infty.\]
    \item If $\;2t = n$, then $J$ is of linear type. Hence, for all $k$, \[\bzero{\A{k}{J}} = \tdp{\A{k}{J}} = -\infty.\]
    \item {\normalfont \cite[2.2]{Huneke3}} If $\;2t = n-1$, then $J$ is of linear type. Hence, for all $k$, \[\bzero{\A{k}{J}} = \tdp{\A{k}{J}} = -\infty.\]
    \item {\normalfont \cite[2.1]{Baetica2}} If $\;\chr K \ne 2$ and $2t = n-2$, then $J$ is of linear type. Hence, for all $k$, \[\bzero{\A{k}{J}} = \tdp{\A{k}{J}} = -\infty.\]
    \item If $\;2t = 4$, then $J$ is of linear type on the punctured spectrum. Hence, for all $k$, \[\tdp{\A{k}{J}} < \infty.\]
    \item If $\;4 < 2t < n-2$, then $J$ is not of linear type on the punctured spectrum. Hence, for some $k$, \[\tdp{\A{k}{J}} = \infty.\]
\end{enumerate}
\end{prop}

\bigskip

In the case of submaximal Pfaffians, explicit resolutions of $J^{k}$ are known thanks to the work of Kustin and Ulrich (see \cref{KUComplexes}). Their computations are much more general than what is given in \cref{BZeroResolutionHeight3Gorenstein}. However, in the case of the submaximal Pfaffians of a generic alternating matrix, we obtain explicit resolutions of $\sym{k}{J}$. Since $J$ is of linear type in this case, $\sym{k}{J} \cong J^{k}$, giving us the desired result. From their work, we obtain the following information on the resolutions.

\bigskip

\begin{cor}
\label[corollary]{BZeroResolutionHeight3Gorenstein}
\hypertarget{BZeroResolutionHeight3Gorenstein}{}
Let $K$ be a field, $n$ an odd integer with $3 \le n$, $X$ an $n \times n$ generic alternating matrix over $R$, $S = \poly{K}{X}$, and $J = \pfaffid{n-1}{X}$. Let $\compp{\mathcal{D}^{k}}$ be the resolution of $J^{k}$ as given in \cref{KUComplexes}. 
\begin{enumerate}
    \item If \,$i \le \min\!\brces{k,n-1}$, then $\bzero{\compp[i]{\mathcal{D}^{k}}} = i$.
    \item If \,$i = k+1 \le n-1$ and $k$ is odd, then $\bzero{\compp[i]{\mathcal{D}^{k}}} = \paren{i-1}+\frac{1}{2}\paren{n-i+1}$.
    \item If \,$i = k+1$ and $k$ is even or if \,$i \ge \min\!\brces{k+2,n}$, then $\bzero{\compp[i]{\mathcal{D}^{k}}} = -\infty$.
\end{enumerate}
\end{cor}

\bigskip

Using the above results, we obtain the following degree bounds.

\bigskip

\begin{thm}[Submaximal Pfaffians]
\label[theorem]{Grade3GorensteinBounds}
\hypertarget{Grade3GorensteinBounds}{}
Let $K$ be a field. Let $R = \poly{K}{x_{1},\ldots,x_{d}}$ be a polynomial ring in $d > 0$ indeterminates over $K$. Suppose $n$ is an odd integer satisfying $3 \le n$, and let $A$ be an $n \times n$ alternating matrix whose entries are all homogeneous elements of $R$ of the same degree $\delta$. Suppose $I = \pfaffid{n-1}{A}$ is of generic alternating height $($that is, $\hgt{I} = 3)$, and suppose $\hgt{\pfaffid{2j}{A}} \ge \min\!\brces{n-2j+2,d}$ for all $j$ in the range $2 \le 2j \le n-3$.
\begin{enumerate}
    \item If \,$k \ge d$ and $d \le n-1$, then 
    \[\bzero{\A{k}{I}} \le \paren{d-1}\paren{\delta-1}, \text{ and }\] 
    \[ \tdp{\A{k}{I}} \le d\paren{\delta-1}.\]
    
    \item If \,$d$ is even, $d \le n-1$, and $k = d-1$, then
    \[\bzero{\A{k}{I}} \le \paren{d-1}\paren{\delta-1}, \text{ and }\] 
    \[ \tdp{\A{k}{I}} \le \paren{d-1}\paren{\delta-1}+\frac{\delta}{2}\paren{n-d+1}-1.\]
    
    \item If \,$d$ is odd and $k = d-1$, then
    \[\bzero{\A{k}{I}} = -\infty, \text{ and }\] 
    \[ \tdp{\A{k}{I}} = -\infty.\]
    
    \item If \,$d \ge n$ or $k \le d-2$, then
    \[\bzero{\A{k}{I}} = -\infty, \text{ and }\] 
    \[ \tdp{\A{k}{I}} = -\infty.\]
\end{enumerate}
\end{thm}

\begin{proof}
Use \cref{DegBoundRemark} together with part (c) of \cref{GlobalConditionsAlternating} and \cref{BZeroResolutionHeight3Gorenstein} as in the proof of \cref{TheoremMaximalMinors}.
\end{proof}

\bigskip

Applying the above results, we draw the following conclusions.

\bigskip

\begin{cor}
\label[corollary]{PfaffianDegreeBoundFacts}
\hypertarget{PfaffianDegreeBoundFacts}{}
Let $K$ be a field. Let $R = \poly{K}{x_{1},\ldots,x_{d}}$ be a polynomial ring in $d > 0$ indeterminates over $K$. Suppose $n$ is an odd integer satisfying $3 \le n$, and let $A$ be an $n \times n$ alternating matrix whose entries are all homogeneous elements of $R$ of the same degree $\delta$. Suppose $I = \pfaffid{n-1}{A}$ is of generic alternating height $($that is, $\hgt{I} = 3)$, and suppose \hspace{0.08mm} $\hgt{\pfaffid{2j}{A}} \ge \min\!\brces{n-2j+2,d}$ for all $j$ in the range $2 \le 2j \le n-3$.
\begin{enumerate}
    \item If \,$d \ge n$, then $I$ is of linear type.
    \item If \,$\delta = 1$, then $I$ is of fiber type.
    \item $\A{k}{I} = 0$ for all $k \le d-2$.
    \item If \,$d$ is odd, then $\A{k}{I} = 0$ for all $k \le d-1$. 
    \item If \,$d$ is odd and $\delta = 1$, then $\paren{x_{1},\ldots,x_{d}}\A{}{I} = 0$. 
\end{enumerate}
\end{cor}

\bigskip

All of these results have already been proven. We restate them here for completeness. In particular, all of these are implications of \cite[6.1.b]{KPU}. Kustin, Polini, and Ulrich used a similar technique to obtain the same results; the difference is a matter of perspective. Kustin, Polini, and Ulrich applied the approximate resolution arguments in settings where the ideal $J$ is of linear type, a simpler context than the method employed in this paper. In the setting of submaximal Pfaffians, the broader perspective reduces to their technique and, therefore, recovers their results.  Before \cite{KPU}, however, some of these results had already been recognized. Both (c) and (d) had been observed by Morey \cite[4.1]{Morey}, also using the complexes in \cref{KUComplexes}. 

Moreover, (a) has been known for the longest amount of time. Indeed, due to the work of J. Watanabe \cite{Watanabe}, we know ideals as in \cref{PfaffianDegreeBoundFacts} are in the linkage class of a complete intersection. Hence, these ideals are strongly Cohen-Macaulay, and therefore, satisfy a condition known as sliding depth \cite[1.4]{Huneke3}. In the presence of the sliding depth condition, a homogeneous ideal is of linear type if and only if the ideal satisfies $G_{\infty}$ \cite{HSV:LinearType}. However, condition (a) implies $G_{\infty}$.

For other values of $2t$, we do not know explicit resolutions of $J^{k}$. However, due to the work of Perlman \cite[Theorem A]{Perlman}, we know the regularity of $J^{k}$ for sufficiently large $k$. 

For the rest of the section, we obtain degree bounds using the regularity in \cite[Theorem A]{Perlman}. As we are not aware of any prior work to obtain degree bounds in the case of Pfaffian ideals of alternating matrices, we believe the results proven in the rest of this section are novel.

\bigskip

\begin{thm}[Size $n-2$ Pfaffians]
\label[theorem]{TheoremN-2Pfaffs}
\hypertarget{TheoremN-2Pfaffs}{}
Let $K$ be a field of characteristic zero. Let $R = \poly{K}{x_{1},\ldots,x_{d}}$ be a polynomial ring in $d > 0$ indeterminates over $K$. Suppose $n$ is an even integer satisfying $4 \le n$, and let $A$ be an $n \times n$ alternating matrix whose entries are all homogeneous elements of $R$ of the same degree $\delta$. Suppose $I = \pfaffid{n-2}{A}$ is of generic alternating height $($that is, $\hgt{I} = 6)$, and suppose \hspace{0.08mm} $\hgt{\pfaffid{2j}{A}} \ge \min\!\brces{\binom{n-2j+2}{2},d}$ for all $j$ in the range $2 \le 2j \le n-4$.

    \begin{enumerate}
        \item If \,$n$ is divisible by \,$4$ and $k \ge n-2$, then 
        \[\bzero{\A{k}{I}} \le \paren{d-1}\paren{\delta-1}+\frac{\delta\paren{n-4}^{2}}{8}, \text{ and }\] 
        \[ \tdp{\A{k}{I}} \le d\paren{\delta-1}+\frac{\delta\paren{n-4}^{2}}{8}.\]
        
        \item If \,$n$ is not divisible by \,$4$ and $k \ge n-2$, then 
        \[\bzero{\A{k}{I}} \le \paren{d-1}\paren{\delta-1}+\frac{\delta\paren{n-2}\paren{n-6}}{8}, \text{ and }\] 
        \[ \tdp{\A{k}{I}} \le d\paren{\delta-1}+\frac{\delta\paren{n-2}\paren{n-6}}{8}.\]
    \end{enumerate}
\end{thm}

\begin{proof}
Apply \cref{DegBoundRemark} with part (d) of \cref{GlobalConditionsAlternating} and \cite[Theorem A]{Perlman}.  
\end{proof}

\bigskip

\begin{cor}
Let $K$ be a field of characteristic zero. Let $R = \poly{K}{x_{1},\ldots,x_{d}}$ be a polynomial ring in $d > 0$ indeterminates over $K$. Suppose $A$ is a $\;6 \times 6$ alternating matrix whose entries are all homogeneous elements of $R$ of degree $\;1$. Let $I = \pfaffid{4}{A}$ be of generic alternating height $($that is, $\hgt{I} = 6)$, and suppose \hspace{0.08mm} $\hgt{\pfaffid{2}{A}} \ge \min\!\brces{15,d}$.

Then $I$ is of fiber type and $\paren{x_{1},\ldots,x_{d}}\A{}{I} = 0$.
\end{cor}

\begin{proof}
Apply \cref{TheoremN-2Pfaffs}. Since $6$ is not divisible by $4$, $\bzero{\A{k}{I}} \le 0$ and $\tdp{\A{k}{I}} \le 0$ for all $k \ge 4$. 

We may then use the Macaulay2 computational software \cite{M2} to compute minimal homogeneous free resolutions of $J^{2}$ and $J^{3}$. These resolutions are linear. Hence, we obtain $\bzero{\A{k}{I}} \le 0$ and $\tdp{\A{k}{I}} \le 0$ when $2 \le k \le 3$ and $\delta = 1$. Therefore, $\bzero{\A{k}{I}} \le 0$ and $\tdp{\A{k}{I}} \le 0$ for all $k$, giving $I$ is of fiber type and $\A{}{I}$ is annihilated by $\paren{x_{1},\ldots,x_{d}}$.
\end{proof}

\bigskip

\begin{thm}[Size $4$ Pfaffians]
\label[theorem]{Theorem4Pfaffs}
\hypertarget{Theorem4Pfaffs}{}
Let $K$ be a field of characteristic zero. Let $R = \poly{K}{x_{1},\ldots,x_{d}}$ be a polynomial ring in $d > 0$ indeterminates over $K$. Suppose $n$ is an  integer satisfying $4 \le n$, and let $A$ be an $n \times n$ alternating matrix whose entries are all homogeneous elements of $R$ of the same degree $\delta$. Suppose $I = \pfaffid{4}{A}$ is of generic alternating height $($that is, $\hgt{I} = \binom{n-2}{2})$. Let $X$ be an $n \times n$ generic alternating matrix over $R$ and $J = \pfaffid{4}{X}$, and suppose \hspace{0.08mm} $\hgt{\pfaffid{2}{A}} \ge \min\!\brces{\binom{n}{2},d}$.

        If $n$ is even and $k \ge n-2$ or if \hspace{0.08mm} $n$ is odd and $k \ge n-3$, then 
        \[\bzero{\A{k}{I}} \le \max\brces{\delta \bzero{\A{k}{J}}, \paren{d-1}\paren{\delta-1}}, \text{ and }\] 
        \[ \tdp{\A{k}{I}} \le \max\brces{\delta \tdp{\A{k}{J}}, d\paren{\delta-1}}.\]
\end{thm}

\begin{proof}
Apply \cref{DegBoundRemark} with part (e) of \cref{GlobalConditionsAlternating} and \cite[Theorem A]{Perlman}.  
\end{proof}

\bigskip

Although we do not know $\bzero{\A{k}{J}}$ in \cref{Theorem4Pfaffs}, one may conjecture $\bzero{\A{k}{J}} \le 0$ for all $k$. This is motivated by the work of Huang, Perlman, Polini, Raicu, and Sammartano \cite{HPPRS} which uses techniques from representation theory to prove the ideal of $2 \times 2$ minors of a generic ordinary matrix (of certain sizes) are of fiber type. Typically, representation theory techniques for ideals of minors of ordinary matrices transfer to the corresponding size Pfaffian ideals of a generic alternating matrix. Therefore, one might suspect $\bzero{\A{k}{J}} \le 0$ for certain values of $n$ in \cref{Theorem4Pfaffs}.

We end with the following general result.

\bigskip

\begin{thm}[Pfaffians of an Alternating Matrix]
\label[theorem]{TheoremGeneralPfaffs}
\hypertarget{TheoremGeneralPfaffs}{}
Let $K$ be a field of characteristic zero. Let $R = \poly{K}{x_{1},\ldots,x_{d}}$ be a polynomial ring in $d > 0$ indeterminates over $K$. Suppose $t$ and $n$ are integers satisfying $4 < 2t < n-2$, and let $A$ be an $n \times n$ alternating matrix whose entries are all homogeneous elements of $R$ of the same degree $\delta$. Let $I = \pfaffid{2t}{A}$ be of generic alternating height $($that is, $\hgt{I} = \binom{n-2t+2}{2})$. Suppose $X$ is an $n \times n$ generic alternating matrix over $R$ and $J = \pfaffid{2t}{X}$, and suppose \hspace{0.08mm} $\hgt{\pfaffid{2j}{A}} \ge \min\!\brces{\binom{n-2j+2}{2},d}$ for all $j$ in the range $2 \le 2j \le 2t-2$.

If \,$n$ is even and $k \ge n-2$ or if \,$n$ is odd and $k \ge n-3$, then 
\[ \bzero{\A{k}{I}} \le  \max\brces{\delta\bzero{\A{k}{J}}, \paren{d-1}\paren{\delta-1}+\delta N\!\paren{t}},  \]
    where $N\!\left(t\right) = \begin{cases}
t\paren{\frac{t}{2}-1} & \text{ if } t \text{ is even}\\
\frac{1}{2}\paren{t-1}^{2} & \text{ if } t \text{ is odd.}
\end{cases}$
\end{thm}

\begin{proof}
Apply \cref{DegBoundRemark} together with part (f) of \cref{GlobalConditionsAlternating} and \cite[Theorem A]{Perlman}.
\end{proof}

\bibliographystyle{mandebib}
\bibliography{CPrefs.bib}
\end{document}